\documentclass[12pt, a4paper]{amsart}

\calclayout

\usepackage{amssymb,mathrsfs,mathtools} % using amsart, so don't need to load amsmath or amsthm
\usepackage{bm}
\usepackage{enumitem}
\usepackage{tensor}
\usepackage{tikz-cd}
\usepackage{hyperref}
\hypersetup{
  colorlinks=true,
  linkcolor={red!50!black},
  filecolor=magenta,
  urlcolor=cyan,
  citecolor={blue!80!black}
}
\usepackage{amsrefs}

\newcommand{\R}{\mathbb{R}}
\newcommand{\Z}{\mathbb{Z}}

\DeclareMathOperator{\Aff}{Aff}
\DeclareMathOperator{\dom}{dom}
\DeclareMathOperator{\Diff}{Diff}
\DeclareMathOperator{\Eff}{Eff}
\DeclareMathOperator{\germ}{germ}
\DeclareMathOperator{\GL}{GL}
\DeclareMathOperator{\id}{id}
\DeclareMathOperator{\loc}{loc}
\DeclareMathOperator{\pr}{pr}

\newcommand{\ol}[1]{\overline{#1}}

\newcommand{\cl}[1]{\mathcal{#1}}
\newcommand{\sr}[1]{\mathscr{#1}}
\newcommand{\fk}[1]{\mathfrak{#1}}
\newcommand{\mbf}[1]{\mathbf{#1}}
\newcommand{\msf}[1]{\mathsf{#1}}

\newcommand{\ti}[1]{\tilde{#1}}

\newcommand{\rra}{\rightrightarrows}
\newcommand{\xar}[1]{\xrightarrow[]{#1}}
\newcommand{\diffloc}{\operatorname{Diff}_{\text{loc}}}
\newcommand{\fiber}[2]{\tensor[_{#1}]{\times}{_{#2}}}

\newcommand{\diffeolqfold}{\mbf{DiffeolQfold}}
\newcommand{\qfoldgrpd}{\mbf{QfoldGrpd}}
\newcommand{\bi}{\mbf{Bi}}
\newcommand{\diffeol}{\mbf{Diffeol}}

\newcommand{\define}[1]{\emph{#1}}
\theoremstyle{plain}
\newtheorem{theorem}     [equation] {Theorem}
\newtheorem{proposition} [equation]  {Proposition}
\newtheorem{lemma}       [equation]  {Lemma}
\newtheorem{corollary}   [equation]  {Corollary}

\newtheorem*{theorem*}              {Theorem}
\newtheorem*{proposition*}          {Proposition}
\newtheorem*{corollary*}            {Corollary}

\theoremstyle{definition}
\newtheorem{definition}  [equation]  {Definition}
\newtheorem*{definition*}           {Definition}

\theoremstyle{remark}

\newtheorem{remark}      [equation]  {Remark}
\newtheorem{example}     [equation]  {Example}
\newtheorem*{remark*}               {Remark}
\newtheorem*{example*}              {Example}
\newtheorem*{question*}             {Question}

\def\eor{\unskip\ \hglue0mm\hfill$\diamond$\smallskip\goodbreak}
\def\eoe{\unskip\ \hglue0mm\hfill$\between$\smallskip\goodbreak}
\def\eod{\unskip\ \hglue0mm\hfill$\diamond$\smallskip\goodbreak}

\numberwithin{equation}{section}

\newenvironment{enumerater}{
  \begin{enumerate}[label = (\roman*)]}{
  \end{enumerate}}

\newif\ifdebug

\debugfalse

\begin{document}
\title{Quasifold Groupoids and Diffeological Quasifolds}



\author{Yael Karshon}
\address{Department of Mathematics, University of Toronto, Toronto, ON, Canada,
and School of Mathematical Sciences, Tel-Aviv University, Tel-Aviv, Israel}
\email{karshon@math.toronto.edu, yaelkarshon@tauex.tau.ac.il}

\author{David Miyamoto}
\address{Department of Mathematics, University of Toronto, Toronto, ON, Canada}
\email{david.miyamoto@mail.utoronto.ca}

\date{\today}
\maketitle

\begin{abstract}
  Quasifolds are spaces that are locally modelled by quotients of $\R^n$ by countable affine group actions. These spaces first appeared in Elisa Prato's generalization of the Delzant construction, and special cases include leaf spaces of irrational linear flows on the torus, and orbifolds. We consider the category of diffeological quasifolds, which embeds in the category of diffeological spaces, and the bicategory of quasifold groupoids, which embeds in the bicategory of Lie groupoids, bibundles, and bibundle morphisms. We prove that, restricting to those morphisms that are locally invertible, and to quasifold groupoids that are effective, the functor taking a quasifold groupoid to its diffeological orbit space is an equivalence of the underlying categories. These results complete and extend earlier work with Masrour Zoghi.
\end{abstract}

\section{Introduction}
\label{sec:introduction}

Quasifolds were introduced by Elisa Prato in \cite{P1999}, as a generalization of manifolds and orbifolds. Whereas manifolds are locally modelled by Cartesian spaces,\footnote{$\R^n$ for some $n$.} and orbifolds are locally modelled by quotients of Cartesian spaces by finite group actions, quasifolds are locally modelled by quotients of Cartesian spaces by countable group actions. These spaces often have very coarse topologies. For example, the irrational tori $T_\alpha := \R / (\Z + \alpha \Z)$ (for $\alpha$ irrational) have trivial quotient topology. In contrast, when we view irrational tori as diffeological spaces, Donato and Iglesias-Zemmour \cite{DIZ1985} proved that $T_\alpha$ and $T_\beta$ are diffeomorphic as diffeological spaces if and only if $\alpha$ and $\beta$ are related by a fractional linear transformation with integer coefficients. \emph{Diffeological quasifolds}, defined by Iglesias-Zemmour and Prato in \cite{IZP2020}, are diffeological spaces that are, at each point, locally diffeomorphic to a quotient space $\R^n/\Gamma$, for a countable group $\Gamma$ acting \emph{affinely}\footnote{see Remark \ref{rem:2}} on $\R^n$. Special cases include orbifolds \cite{IZKZ2010}, and irrational tori. The groups $\Gamma$ may change from point to point. As diffeological spaces, diffeological quasifolds inherit a notion of smooth maps and a de Rham complex of differential forms.

A ``higher'' approach to orbifolds is to define them as Lie groupoids that are, at each point, locally isomorphic to the restriction of an action groupoid $\Gamma \ltimes \R^n$, for a finite group $\Gamma$ acting linearly on $\R^n$. We similarly define a \emph{quasifold groupoid} to be a Lie groupoid that is, at each point, locally isomorphic to the restriction of the action groupoid $\Gamma \ltimes \R^n$, for a countable group $\Gamma$ acting affinely on $\R^n$. The groups $\Gamma$ may change from point to point.

Our main result is that the categories of diffeological quasifolds and quasifold groupoids are equivalent, after restricting to local isomorphisms and effective quasifold groupoids. This completes and extends earlier work about orbifolds by Masrour Zoghi and Yael Karshon \cite{Z2010}. We place quasifold groupoids in the bicategory $\bi$,\footnote{using Lerman's notation in \cite{L2010}} whose objects are Lie groupoids, whose arrows are bibundles, and whose 2-arrows are morphisms of bibundles. We introduce the notion of a \emph{locally invertible} bibundle (Definition \ref{def:17}), and the sub-bicategory of effective quasifold groupoids $\qfoldgrpd^{\text{loc-iso}}_{\text{eff}}$,  whose objects are effective quasifold groupoids, whose arrows are locally invertible bibundles, and whose 2-arrows are morphisms of bibundles. We place diffeological quasifolds in the category $\diffeol$ of diffeological spaces, whose objects are diffeological spaces, and whose arrows are diffeologically smooth maps. We view this as a bicategory with the identity 2-arrows. We introduce the sub-bicategory of diffeological quasifolds $\diffeolqfold^{\text{loc-iso}}$, whose objects are diffeological quasifolds, whose arrows are local diffeomorphisms, and with identity 2-arrows.

There is a quotient functor of bicategories $\mbf{F}:\bi \to \diffeol$, that takes a groupoid to its orbit space. Our main result is Theorem \ref{thm:1}:
\begin{theorem*}
  The quotient functor $\mbf{F}$ restricts to a functor of bicategories $\mbf{F}_{\text{Quas}}:\qfoldgrpd^{\text{loc-iso}}_{\text{eff}} \to \diffeolqfold^{\text{loc-iso}}$ that is: essentially surjective, surjective on arrows, and injective on arrows up to 2-isomorphism.
\end{theorem*}
This applies \emph{mutatis mutandis} to orbifold groupoids and diffeological orbifolds.

By identifying isomorphic bibundles of groupoids, we form the Hilsum-Skandalis category $\mbf{HS}$ of Lie groupoids, whose objects are Lie groupoids, and whose arrows are isomorphism classes of bibundles. Restricting to quasifolds, and using the same notation, the theorem above yields
\begin{corollary*}
  The functor $\mbf{F}_{\text{Quas}}$ gives an equivalence of categories between $\qfoldgrpd^{\text{loc-iso}}_{\text{eff}}$ (viewed in $\mbf{HS}$) and $\diffeolqfold^{\text{loc-iso}}$.
\end{corollary*}

In Iglesias-Zemmour and Prato's recent work \cite{IZP2020}, which builds on \cite{IZL2018}, the authors define diffeological quasifolds, and associate to each diffeological quasifold a groupoid, and then a $C^*$ algebra, that is unique up to Muhly-Renault-Williams equivalence. Their construction implies that the quotient functor $\mbf{F}_{\text{Quas}}$ is essentially surjective, and is full on isomorphisms. They do not view their groupoids as Lie groupoids, nor do they introduce the notion of quasifold groupoids.

Prato originally introduced quasifolds through symplectic geometry in \cite{P2001}, in order to generalize the Delzant construction. Recent work involving quasifolds and symplectic geometry includes \cite{BP2018}, \cite{BP2019}, \cite{BPZ2019}, and \cite{LS2019}. Hoffman \cite{H2020} works with not-necessarily-effective quasifold groupoids as stacks. Our results can also be written in terms of (effective) stacks, but we leave this for another paper. For orbifolds as diffeological spaces, we already mentioned \cite{IZKZ2010} and \cite{IZL2018}.  A thorough comparison of the categories of orbifolds as diffeological spaces, orbifold groupoids, and orbifolds as Sikorski differential spaces, can be found in \cite{W2017}.

Each section of this article addresses a different component of the main theorem. In Section \ref{sec:diff-lie-group}, we review diffeology, and introduce diffeological quasifolds. In Section \ref{sec:lie-groupoids}, we review Lie groupoids, introduce locally invertible bibundles, and introduce quasifold groupoids. In Section \ref{sec:from-diff-quas}, we show that the quotient functor $\mbf{F}_{\text{Quas}}$ is essentially surjective. In Section \ref{sec:lift-local-isom} we show that the quotient functor $\mbf{F}_{\text{Quas}}$ is surjective on arrows. In Section \ref{sec:an-equiv-categ}, we show that the quotient functor $\mbf{F}_{\text{Quas}}$ is injective on arrows up to 2-isomorphism. In the last section, Section 7, we describe two effective actions of $\Z$ on $\R$ whose orbits coincide, but whose action groupoids are not Morita equivalent. Namely, they are not related by an invertible bibundle. Thus, our results do not extend to arbitrary countable group actions.

\ifdebug \newpage \fi

\subsection*{Acknowledgement}
D. Miyamoto thanks Eugene Lerman and Rui Loja Fernandes for helpful comments on the bicategory of Lie groupoids, and also thanks his hosts at Tel Aviv University for the month he visited there. Y. Karshon thanks Itai Bar-Natan, David Metzler, Andre Henriques, Eugene Lerman, and Hans Duistermaat, for discussions that contributed to the earlier version of this work. Both authors are indebted to Masrour Zoghi, who set the groundwork and initial direction for this project.

\subsection*{Funding}
This work was partly funded by the Natural Sciences and Engineering Research Council of Canada, and by a Graduate Student Travel Grant provided by the Department of Mathematical and Computational Sciences at the University of Toronto Mississauga.

\section{Diffeological Quasifolds}
\label{sec:diff-lie-group}

\subsection{Diffeology}
\label{sec:diffeology}

For the reader unfamiliar with diffeology, we review it here.

\begin{definition}[Diffeology]
  \label{def:1}
  Let $X$ be a set. A \define{parametrization} into $X$ is a map from an open subset of a Cartesian space into $X$. A \define{diffeology} on $X$ is a set $\cl{D}$ of parametrizations, whose members are called \define{plots}, such that
  \begin{itemize}
  \item constant maps are plots;
  \item if a parametrization $p:\msf{U} \to X$ is such that about each $r \in \msf{U}$, there is an open $\msf{V} \subseteq \msf{U}$ and a plot $q:\msf{V} \to X$ such that $p = q|_{\msf{V}}$, then $P$ is a plot;
    \item if $p:\msf{U} \to X$ is a plot and $\msf{V}$ is an open subset of a Cartesian space, then for any smooth $F:\msf{V} \to \msf{U}$, the pre-composition $F^*p$ is a plot.
    \end{itemize}
    A set equipped with a diffeology is a \define{diffeological space}.
    \eod  
  \end{definition}

  % \begin{remark}
  % \label{rem:1}
  %   Using sheaves, we can equivalently define a diffeological space as a concrete $X:\mbf{Open} \to \operatorname{Set}$, where $\mbf{Open}$ is the concrete site on the category of open subsets of Cartesian spaces, and smooth maps between them. See \cite{BH2011} for this point of view.
  %   \eor
  % \end{remark}

  The set of locally constant parametrizations into $X$, and the set of all parametrizations into $X$, are both diffeologies, called respectively \emph{discrete} and \emph{coarse}. Every other diffeology sits between these two. A classical smooth manifold\footnote{A smooth manifold is a topological space equipped with a maximal smooth atlas that, unless we say otherwise, is Hausdorff and second-countable.} $M$ carries a canonical diffeology $\cl{D}_M$ consisting of the smooth maps (in the usual sense) from subsets of Cartesian spaces into $M$.

  \begin{definition}[Smooth maps]
    \label{def:2}
    We say a map $f:X \to Y$ between diffeological spaces is \define{(diffeologically) smooth} if for every plot $p$ of $X$, the pullback $p^*f$ is a plot of $Y$. Denote the set of smooth maps from $X$ to $Y$ by $C^\infty(X, Y)$.
    \eod
\end{definition}

If $X$ is discrete, or $Y$ is coarse, all maps $X \to Y$ are smooth. A map between classical manifolds is diffeologically smooth if and only if it is smooth in the classical sense. Diffeological spaces and smooth maps between them form a category.

\begin{definition}[Category of diffeological spaces]
  \label{def:3}
  The category $\diffeol$ has objects diffeological spaces, and arrows smooth maps between them. When we need to view $\diffeol$ as a bicategory, we simply add the identity 2-arrows.
  \eod
\end{definition}

Assigning to each classical smooth manifold $M$ the diffeological space $(M, \cl{D}_M)$ defines a full and faithful functor from the category of classical manifolds, with their smooth maps, into $\diffeol$.

We will require a notion of locality for diffeological spaces.

\begin{definition}[D-topology]
  \label{def:4}
  The \define{D-topology} on a diffeological space $X$ is the finest topology in which all plots are continuous. Equivalently, $U \subseteq X$ is D-open if and only if $p^{-1}(U)$ is open for all plots $p$.
  \eod
\end{definition}

Every smooth map is continuous in the D-topology. The D-topology of a classical manifold $M$, viewed as a diffeological space, is its manifold topology.

Every subset of a diffeological space inherits a subset diffeology:

\begin{definition}[Subset diffeology]
  \label{def:5}\
  For a subset $S$ of a diffeological space $X$, with inclusion denoted $\iota: S\hookrightarrow X$, the \define{subset diffeology} on $S$ consists of all parametrizations $p:U \to S$ such that $\iota \circ p$ is a plot of $X$.
  \eod
\end{definition}

 Given a diffeological space $X$ and a subset $S$, the D-topology of the subset diffeology on $S$ is contained in the subset topology that $S$ inherits from the D-topology on $X$. When $S$ is a D-open subset of $X$, these topologies coincide. We use the D-topology and subset diffeology to define local diffeomorphisms:

\begin{definition}[Local diffeomorphisms]
  \label{def:6}\
  A map $f:X \to Y$ between diffeological spaces is a \define{local diffeomorphism} if, for each $x \in X$, the map restricts to a diffeomorphism between some D-open neighbourhoods of $x$ and $f(x)$. Denote the set of local diffeomorphisms by $\Diff_{\loc}(X, Y)$.

  We call a diffeomorphism $f:U \to U'$, between D-open subsets $U \subseteq X$ and $U' \subseteq Y$, a \define{transition} from $X$ to $Y$. We sometimes write $f:X \dashrightarrow Y$.
  \eod
\end{definition}

We emphasize that local diffeomorphisms are globally defined, and transitions are locally defined.

Finally, a diffeology passes to quotients:

\begin{definition}[Quotient diffeology]
  \label{def:7}
 Given a diffeological space $X$, and equivalence relation $\cl R$, with quotient map $\pi:X \to X/\cl R$, the \define{quotient diffeology} on $X/\cl R$ consists of those parametrizations $p:\mathsf{U} \to X/\cl R$ such that about each $r \in \mathsf{U}$ there is an open $\mathsf{V}\subseteq \mathsf{U}$ and a plot $q:\mathsf{V} \to X$ such that $p|_{\mathsf{V}} = \pi \circ q$. In a diagram,
 \begin{equation*}
   \begin{tikzcd}
     & & X \ar[d, "\pi"] \\
     r \in \mathsf{V} \ar[rru, "\exists q"] \ar[r, hook] & \mathsf{U} \ar[r, "p"] & X/\cl{R}.
   \end{tikzcd}
 \end{equation*}
 \eod
\end{definition}

The D-topology of the quotient diffeology is the quotient topology induced by the D-topology on $X$.

\subsection{Diffeological quasifolds}
We now define diffeological quasifolds. We use Iglesias-Zemmour and Prato's definition from \cite{IZP2020}. It is similar to the diffeological orbifolds introduced by Iglesias-Zemmour, Karshon, and Zadka in \cite{IZKZ2010}.

\begin{definition}[Diffeological quasifolds]
\label{def:8}
A \define{diffeological} $n$-\define{quasifold} is a second-countable diffeological space $X$ such that, for each $x \in X$, there is a D-open neighbourhood $U$ of $x$, a countable subgroup $\Gamma$ of affine transformations of $\R^n$, an open, $\Gamma$-invariant subset $\msf{V} \subseteq \R^n$, and a diffeological diffeomorphism (a \define{chart}) $F:U \to \msf{V}/\Gamma$. Here, $\msf{V}/\Gamma$ is equipped with its quotient diffeology, which coincides with the subset diffeology induced from $\R^n/\Gamma$, (c.f.\ Lemma \ref{lem:5}).

We call such  $\msf{V}/\Gamma$ a \define{model diffeological quasifold}, and we call a collection $\{F:U \to \msf{V}/\Gamma\}$ of diffeomorphisms whose domains $U$ are an open cover of $X$ a \define{(diffeological quasifold) atlas} for $X$.
  \eod
\end{definition}

\begin{remark}
  \label{rmk:1}
To be a diffeological quasifold is a local condition: given a diffeological quasifold $X$, for every $x \in X$ and open neighbourhood $U'$ of $x$, there is an open neighbourhood $U$ of $x$ contained in $U'$ such that $U$ is diffeomorphic to a model quasifold $\msf{V}/\Gamma$. Compare to Remark \ref{rem:9}.

  Denoting the quotient by $\pi:\R^n \to \R^n/\Gamma$, no generality is gained by writing our models as $\pi(\msf{V})$ for arbitrary open subsets $\msf{V}$ of $\R^n$. See Lemma \ref{lem:5}, and compare this to Remark \ref{rmk:2}. Furthermore, no generality is gained if we assume $\Gamma$ is merely a countably group acting affinely on $\R^n$  (i.e.\ if we assume the action homomorphism $\Gamma \to \Aff(\R^n)$ is not necessarily injective). This is because $\msf{V}/\Gamma = \msf{V}/(\Gamma/\ker(\Gamma))$, where $\ker(\Gamma)$ denotes the subgroup of those $\gamma \in \Gamma$ that act as the identity. Compare this to Remark \ref{rmk:2}. 
  \eor
\end{remark}

\begin{definition}[Category of diffeological quasifolds]
  \label{def:9}
  The category $\diffeolqfold$ is the subcategory of $\diffeol$ whose objects are diffeological quasifolds and whose morphisms are smooth maps between them. Restricting morphisms to include only local diffeomorphisms, we get $\diffeolqfold^{\text{loc-iso}}$.
  \eod
\end{definition}

\begin{example}
  \label{ex:1}
  The \emph{irrational tori} are important examples of quasifolds. For an irrational $\alpha \in \R$, the irrational torus $T_\alpha$ is the diffeological quotient space $\R / (\Z + \alpha \Z)$. Here the group $\Z + \alpha \Z$ is countable, and acts affinely on $\R$ by addition, hence $T_\alpha$ is a model diffeological quasifold. Iglesias-Zemmour and Donato in \cite{DIZ1985} prove that $T_\alpha \cong T_\beta$ if and only if $\alpha$ and $\beta$ are related by a fractional linear transformation with integer coefficients. On the other hand, irrational tori are trivial as topological spaces, and thus any two are homeomorphic. In Example \ref{ex:7}, we illustrate the groupoid picture.
  \eoe
\end{example}

\begin{example}
\label{ex:2}
If the groups $\Gamma$ are all finite subgroups of $\GL(\R^n)$, the corresponding quasifold is a \emph{diffeological orbifold} as defined in \cite{IZKZ2010}. Its D-topology need not be Hausdorff. By Palais' slice theorem \cite{P1961}, the quotient space of a locally proper Lie group action with finite isotropy groups is a diffeological orbifold. More generally, the quotient space of a locally proper Lie groupoid with finite isotropy groups is a diffeological orbifold, c.f.\ Example \ref{ex:8}.
\eoe
\end{example}

\begin{remark}
  \label{rem:2}
  Our model quasifolds, namely diffeological spaces $\msf{V} / \Gamma$ for countable groups $\Gamma$ acting affinely on $\msf{V} \subseteq \R^n$, are the same as in \cite{IZP2020}, but differ from Prato's original models \cite{P1999}. In \cite{P1999}, Prato models quasifolds by the topological quotient spaces $M/\Gamma$, where $M$ is a connected, simply-connected manifold, and $\Gamma$ is a discrete group acting smoothly, such that the set of points where the action is free is dense, open, and connected. It is not clear whether these models are equivalent. Nevertheless, we expect the theory that we develop to work equally well for Prato's original quasifolds, because the two key lemmas below still hold.
\eor
\end{remark}

The two key lemmas below concern lifting properties of maps into model quasifolds. The first is a variant of a theorem in Section 3 of \cite{IZP2020}, which is similar to Lemma 17 from \cite{IZKZ2010}. Its proof uses the same technique found in Lemma 5.8 of \cite{Miy2022} and Proposition 5.4 of \cite{KW2016}.
\begin{lemma}
\label{lem:1}
  Suppose $\Gamma$ is a countable group acting affinely on $\R^n$, and $\msf{W}$ is a connected open subset of $\R^n$, and $h:\msf{W} \to \R^n$ is a $C^1$ map preserving $\Gamma$-orbits. Then for some $\gamma \in \Gamma$, of the form $\gamma \cdot x = A_\gamma x + b_\gamma$,
  \begin{equation*}
    h(x) = \gamma \cdot x = A_\gamma x + b_\gamma, \quad \text{for all } x \in \msf{W}. 
  \end{equation*}
\end{lemma}
Before proving this lemma, we state a convenient corollary. Recall that two smooth functions $f,g:M \to N$ between manifolds have the same \define{germ} at $x \in M$ if there is a neighbourhood $U$ of $x$ for which $f|_U = g|_U$; we denote by $\germ_x f$ the equivalence class of functions with the same germ at $x$ as $f$.
\begin{corollary}
  \label{cor:1}
  Suppose $\Gamma$ is a countable group acting affinely on $\R^n$, and $\msf{W}$ is a (not necessarily connected) open subset of $\R^n$, and $h:\msf{W} \to \R^n$ is a $C^1$ map preserving $\Gamma$-orbits. Then for each $x \in W$, there is some $\gamma \in \Gamma$ such that $\germ_x f = \germ_x \gamma$.
\end{corollary}

\begin{proof}[Proof of Lemma \ref{lem:1}]
  For each $\gamma \in \Gamma$, set $\Delta_\gamma := \{x \in \msf{W} \mid h(x) = \gamma \cdot x\}$, and denote its interior by $\Delta_\gamma^\circ$. Since $h$ preserves $\Gamma$-orbits, $\msf{W} = \bigcup_\Gamma \Delta_\gamma$. Furthermore, each $\Delta_\gamma$ is closed in $\msf{W}$, being the pre-image of the diagonal in $\msf{W} \times \msf{W}$ under the continuous map $x \mapsto (h(x), \gamma \cdot x)$. Therefore, as a consequence of the Baire category theorem, $\bigcup_\Gamma \Delta_\gamma^\circ$ is dense in $\msf{W}$.

  The function $h$ is affine when restricted to each $\Delta_\gamma^\circ$, so $Dh = A_\gamma$ on this subset. Thus,
  \begin{equation*}
    Dh\left(\bigcup \Delta_\gamma^\circ\right) = \{A_\gamma \mid \gamma \in \Gamma, \text{ and }\Delta_\gamma^\circ \neq \emptyset\}
  \end{equation*}
  is discrete. But $Dh$ is continuous, and therefore
  \begin{equation*}
    \ol{Dh\left(\bigcup \Delta_\gamma^\circ \right)} = \ol{Dh \left(\ol{\bigcup \Delta_\gamma^\circ}\right)} = \ol{Dh(\msf{W})}
  \end{equation*}
  is discrete. Since $\msf{W}$ is connected, this set must be a singleton, and $Dh = A_\gamma$ for some $\gamma \in \Gamma$ on all of $\msf{W}$. Thus, only the $\Delta_{(A_\gamma, b_\gamma)}^\circ$ are potentially non-empty.

  The difference $h - A_\gamma$ is continuous on $\msf{W}$, and restricts to $b_\gamma$ on each $\Delta_{(A_\gamma, b_\gamma)}^\circ$. By the exact same argument as before, we conclude that precisely one $\Delta_{(A_\gamma, b_\gamma)}^\circ$ is non-empty, and the corresponding $\gamma$ is the desired element of $\Gamma$.
\end{proof}

This next lemma appears as a theorem in Section 4 of \cite{IZP2020}. It is similar to Lemma 23 in \cite{IZKZ2010}, where they give a global result for finite group actions. See also Lemmas 1.6 and 1.7 in \cite{P2001}.

\begin{lemma}
\label{lem:2}
Suppose $f:\msf{V}/\Gamma \to \msf{V}'/\Gamma'$ is a local diffeomorphism between model diffeological quasifolds. For every $\bm{x} \in \msf{V}/\Gamma$ and every $r \in \bm{x}$ and $r' \in f(\bm{x})$, there is a transition $\ti f$ between open subsets of $\msf{V}$ and $\msf{V}'$ taking $r$ to $r'$ and lifting $f$.
\end{lemma}

\begin{proof}
  Denote the quotient maps by $\pi$ and $\pi'$, respectively. We first show there is a local lift taking $r$ to $r'$. The map $f \pi:\msf{V} \to \msf{V}'/\Gamma'$ is diffeologically smooth, and therefore admits a local lift $\widetilde{f\pi}$ about $r$. Both $r'$ and $\widetilde{f\pi}(r)$ must be in the same orbit (precisely, in $f(\bm{x})$), therefore there is some $\gamma' \in \Gamma'$ for which $\gamma' \cdot \widetilde{f \pi}(r) = r'$. Then $\gamma' \circ \widetilde{f\pi}$ is the desired local lift of $f$. This is illustrated in the following diagram (the dashed arrow indicates ``locally defined''):
  \begin{equation*}
    \begin{tikzcd}
      (\msf{V}, r) \ar[d, "\pi"] \ar[r, dashed, "\widetilde{f\pi}"] \ar[dr, "f\pi"] & (\msf{V}', \widetilde{f\pi}(r)) \ar[d, "\pi'"] \ar[r, "\gamma'"] & (\msf{V}',r') \ar[dl, "\pi'"] \\
     (\msf{V}/\Gamma, \bm{x}) \ar[r, "f"] & (\msf{V}'/\Gamma', f(\bm{x})). &
    \end{tikzcd}
  \end{equation*}
  Now fix a local lift $\ti{f}:(\msf{V},r) \dashrightarrow (\msf{V}',r')$ of $f$ taking $r$ to $r'$. By assumption on $f$, there are open neighbourhoods $U$ of $\bm{x}$ and $U'$ of $\bm{x}'$, such that the restriction $f:U \to U'$ is a diffeomorphism. By exactly the same procedure as above, we may lift $f^{-1}:U' \to U$ to a locally-defined function $\ti{s}:(\msf{V}',r') \dashrightarrow (\msf{V},r)$.
  
  Consider the composition $\ti{s} \ti{f}$, which is defined on some neighbourhood of $r$. It preserves $\Gamma$-orbits, because both $\ti{f}$ and $\ti{s}$ are lifts of $f$ and $f^{-1}$, respectively. Therefore, by Corollary \ref{cor:1}, there is some $\gamma \in \Gamma$ such that $\germ_r\ti{s} \ti{f} = \germ_r \gamma$. In particular, differentiating yields $D\gamma_r = D\ti{s}_{r'} D\ti{f}_r$, so $D\ti{f}_r$ has a left inverse. An exactly similar argument applied to $\tilde{f} \tilde{s}$ shows that $D\tilde{f}_r$ has a right inverse. By the inverse function theorem, we conclude $\ti{f}$ restricts to a transition taking $r$ to $r'$.
\end{proof}

\ifdebug \newpage \fi

\section{Quasifold Groupoids}
\label{sec:lie-groupoids}
In this section, we introduce quasifold groupoids.
\subsection{Bicategory of Lie groupoids}
\label{sec:hils-skand-categ}

 We work inside the bicategory Lie groupoids, denoted $\bi$, whose objects are Lie groupoids, whose arrows are bibundles, and whose 2-arrows are bibundle isomorphisms. For a thorough treatment, see Lerman's review \cite{L2010}, or Moerdijk and Mrc\u{u}n's book \cite{MM2003}. While the material in this subsection is standard, we take a less common local-first approach to the definition of bibundles. Our definition of locally invertible bibundle is new. At the end of this subsection, we introduce the functor $\mbf{F}:\bi \to \diffeol$.

A \define{Lie groupoid} is a small category $G \rra G_0$ with invertible arrows, such that the base $G_0$ is a (Hausdorff and second-countable) smooth manifold, the arrow space $G$ is a not necessarily Hausdorff nor second-countable smooth manifold, all structure maps are smooth, and furthermore the source map $s$ (hence the target map $t$) is a smooth submersion with Hausdorff fibers. We write the composition of arrows $g: x \mapsto y$ and $g':y \mapsto z$ as $g'g :x \mapsto z$. If $\iota: \cl{U} \to G_0$ is a submersion, we may form the pullback groupoid $\iota^*G \rra \cl{U}$, whose arrows from $x$ to $y$ are the arrows in $G$ from $\iota(x)$ to $\iota(y)$. If $U$ is an open subset of $G_0$, we denote the corresponding pullback groupoid by $G|_U$, and identify its arrow space with $s^{-1}(U) \cap t^{-1}(U)$. We say that $G$ is \define{\'{e}tale} if $s$ (hence $t$) is a local diffeomorphism. We say that $G$ is \emph{proper} if its arrow space is Hausdorff and the map $(s,t):G \to G_0 \times G_0$ is proper. By an \emph{isomorphism} of Lie groupoids we mean a smooth functor $F:G \to H$ such that there exists a smooth functor $F^{-1}:H \to G$ for which $FF^{-1}$ and $F^{-1}F$ are the identity functors. The \define{orbit space} of a Lie groupoid $G$ is the quotient of $G_0$ by this relation: $x \sim y$ if there is an arrow $x \mapsto y$. We denote the orbit space of $G$ by $|G|$ or $G_0/G$.

The full definition of a bicategory is rather technical, and we refer the reader to \cite{L2010} for an introduction. For our purpose, a \define{bicategory} consists of objects, arrows between objects, and 2-arrows between arrows. These must satisfy some compatibility and composition conditions, which we do not write. The objects and arrows almost form a category, except that composition of arrows need not be associative. Instead, for any composable $f,g,h$, there exists an invertible 2-arrow $\alpha:f(gh) \to (fg)h$. The arrows and 2-arrows form a category. 

A first example of a Lie groupoid is an action groupoid. Given a Lie group $G$ action on a manifold $M$ (from the left), the \define{action groupoid} $G \ltimes M$ has arrow space $G \times M$ and base space $M$. The source and target maps are $s(g,x) = x$ and $t(g,x) = g \cdot x$.

\begin{definition}[Right actions]
\label{def:10}
  A \define{right action} of a Lie groupoid $H\rra H_0$ on a manifold $P$ consists of maps $\mu:P \fiber{a}{t} H \to P$ and $a:P \to H_0$. We call $a$ the \define{anchor}, and $\mu$ the \define{multiplication}, and denote $\mu(p, h)$ by $p \cdot h$. We have the following Cartesian square on the left, and require the following commuting diagram of smooth functions on the right
  \begin{equation*}
    \begin{tikzcd}
      P \fiber{a}{t} H \ar[r, "\pr_2"] \ar[d, "\pr_1"] & H \ar[d, "t"] \\
      P \ar[r, "a"] & H_0
    \end{tikzcd}
    \qquad
    \begin{tikzcd}
      P \fiber{a}{t} H \ar[r, "\pr_2"] \ar[d, "\mu"] & H \ar[d, "s"] \\
      P \ar[r, "a"] & H_0.
    \end{tikzcd}
  \end{equation*}
 Furthermore we require
  \begin{itemize}
  \item $(p \cdot h)\cdot h' = p \cdot (hh')$ whenever this makes sense, and
  \item $p \cdot 1_{a(p)} = p$ for all $p \in P$.
  \end{itemize}

  Given a right $H$ action, we may form the \define{action groupoid} $P \rtimes H$, with arrows $P \tensor[_a]{\times}{_t} H$, base $P$, source $\mu$, target $\pr_1$, and multiplication of arrows $(p, h)(p', h') = (p, hh')$.  Furthermore, $a$ and $\pr_2$ assemble into a smooth functor of Lie groupoids $P \rtimes H \to H$.
  \eod
\end{definition}

% We can view $a$ as ``$H$-equivariant.'' Given an arrow $h:y \to x$, we write $x \cdot h = y$. Then
% \begin{equation*}
%   a(p \cdot h) = s(h) = x \cdot h = a(p) \cdot h.
% \end{equation*}

\begin{remark}
  \label{rem:3}\
  \begin{enumerater}
  \item If $\cl{O} \subseteq P$ is an open, $H$-invariant subset of $P$, then $H$ acts on $\cl{O}$, with the same anchor $a:\cl{O} \to H_0$ and the same multiplication $\mu:\cl{O} \fiber{a}{t} H \to \cl{O}$.
    \item If $V \subseteq H_0$ is open, then $a^{-1}(V)$ need not be $H$-invariant, but we still have an action of $H|_V$ on $a^{-1}(V)$, with the same anchor $a:a^{-1}(V) \to V$, and the same multiplication. 
  \end{enumerater}
 
\end{remark}

Fix a Lie groupoid $H$. A principal $H$-bundle over a manifold $B$ will consist of a manifold $P$, a right $H$ action on $P$, and a map $\pi:P \to B$, satisfying certain axioms. We will denote such a principal bundle by $P \xar{\pi} B$. We will formulate these axioms using the notion of a trivial bundle.

\begin{definition}
  \label{def:11}
  Fix a smooth manifold $M$, and a map $\phi:M \to H_0$. The \define{trivial} $H$-\define{bundle over} $\phi$ consists of the space $P := M \fiber{\phi}{t} H$, equipped with the right $H$ action on $M \fiber{\phi}{t} H$ that is given by the anchor map $s \circ \pr_2 :M \fiber{\phi}{t} H \to H_0$ and the multiplication $(m, h) \cdot h' := (m, hh')$, and equipped with the projection $M \fiber{\phi}{t} H \xar{\pr_1} M$.
  \eod
\end{definition}

Note that in a trivial $H$-bundle, the fibers of the projection map $P \xar{\pi} M$ are the $H$-orbits, and the $H$ action on $P$ is free, meaning that if $p \cdot h = p \cdot h'$, then $h = h'$.  The following example motivates our definition.
\begin{example}
  \label{ex:3}
  If we take $M = H_0$, and $\phi:H_0 \to H_0$ the identity map, identifying $P = H_0 \fiber{\phi}{t} H$ with the arrow space $H$,  we find that the trivial $H$-bundle over the identity map is the right action of $H$ on its arrow space along the anchor $s$, given by right multiplication. The trivial $H$-bundle over an arbitrary $\phi:M \to H_0$ gives a Cartesian square,
  \begin{equation*}
  \begin{tikzcd}
    M \fiber{\phi}{t} H \ar[r, "\pr_2"] \ar[d, "\pr_1"] & H \ar[d, "t"] \\
    M \ar[r, "\phi"] & H_0
  \end{tikzcd}
\end{equation*}
in which the top map is $H$-equivariant, and the vertical maps are $H$-invariant. If we take $H$ to be a Lie group, viewed as a Lie groupoid over a point $H \rra \{pt\}$, then the right action of $H$ on itself is the usual one given by right multiplication, and the trivial $H$-bundle over $M \to \{pt\}$ is a usual trivial $H$-bundle $M \times H \to M$.
\eoe
\end{example}

% \begin{remark}
% \label{rem:4}
%   More generally, given a right $H$ action on $P$ along $a$, a smooth $H$-invariant map $P \xar{\pi} H_0$, and a smooth map $M \xar{\phi} H_0$, we get an $H$ action on the pullback $M \tensor[_\phi]{\times}{_a} P$. The anchor is $a \pr_2$, and the multiplication is $(m, p) \cdot h = (m, p\cdot h)$. Note that the projection $\pr_1$ is $H$-invariant. For future reference, we will denote this pullback data by $\phi^*(P \xar{\pi} H_0)$.
% \end{remark}

\begin{definition}[Principal bundles]
\label{def:12}
  Fix a smooth manifold $B$ and a Lie groupoid $H$. A \emph{principal} $H$-\emph{bundle over }$B$ consists of a manifold $P$, a right $H$ action on $P$ along the anchor $a$, and an $H$-invariant map $P \xar{\pi} B$, such that the following holds: about every $b \in B$, there exist an open neighbourhood $U$, a section $\sigma:U \to P$ of $\pi$, and an $H$-equivariant diffeomorphism $\Phi$ fitting in the commutative diagram\footnote{Here the action of $H$ on $P|_U := \pi^{-1}(U) \subseteq P$ is induced by $H$ acting on $P$, according to Remark \ref{rem:3} (i); the set $P|_U$ is $H$-invariant by our assumption on $\pi$.}

  \begin{equation}
    \label{eq:1}
    \begin{tikzcd}
      \pi^{-1}(U) \ar[r, "\Phi"] \ar[d, "\pi" ] &U \fiber{a\sigma}{t} H \ar[dl, "\pr_1"] \\
      U.
    \end{tikzcd}
  \end{equation}
  In other words, $P|_U := \pi^{-1}(U)$ is isomorphic to the trivial $H$-bundle over $a\sigma:U \to H_0$.
  \eod
\end{definition}
For a Lie groupoid $G$, we can similarly define left $G$ actions, and left principal $G$-bundles.

\begin{remark}
\label{rem:6}
Other authors, such as \cite{L2010} and \cite{MM2003}, define a principal right $H$-bundle to be a manifold $P$ with right $H$ action, and an $H$-invariant surjective submersion $P \xar{\pi} B$, such the map
\begin{equation*}
      P \tensor[_a]{\times}{_t} H \to P \tensor[_\pi]{\times}{_\pi} P, \quad (p, h) \mapsto (p, p\cdot h)
\end{equation*}
is a diffeomorphism. This definition is equivalent to ours.
\eor
\end{remark}

% As a brief aside, we note that, unlike in our definition of trivial bundles, an arbitrary principal $H$ bundle $P \xar{\pi} B$ does not come with a map $B \to H_0$. Instead, each local trivialization of $P$ gives a map $U \to H_0$ for an open subset $U$ of $B$. This map is $a\sigma$.

\begin{example}
  \label{ex:4}
  Every trivial $H$-bundle over $\phi:B \to H_0$ is a principal $H$-bundle: take $\sigma: B \to B \tensor[_\phi]{\times}{_t} H$ to be $\sigma(b) := (b,1_{\phi(b)} )$, note that $a\sigma = \phi$, and take $\Phi$ to be the identity map on $B \fiber{\phi}{t} H$. In particular, $H \xar{t} H_0$, where $H$ acts on its arrows by right multiplication, is a principal $H$-bundle, with $\sigma(y) = 1_y$. If $H$ is a Lie group viewed as a Lie groupoid over a point, then a principal $H$-bundle agrees with the usual notion of a principal $H$ group bundle.
  \eoe
\end{example}

We can pull back principal $H$-bundles by smooth maps. We will need this construction to associate a bibundle to a smooth functor.
\begin{definition}[Pullback bundle]
  \label{def:13}
 The \define{pullback} of a principal right $H$-bundle $P \xar{\pi}  B$ by a smooth map $f:M \to B$, denoted $f^*(P \xar{\pi} B)$, is the principal right $H$-bundle consisting of the manifold $M \fiber{f}{\pi} P$, the right $H$ action along the anchor map $a \circ \pr_2$ given by $(m,p) \cdot h := (m, p \cdot h)$, and the projection $ M \fiber{f}{\pi} P \xar{\pr_1} M$.
  \eod
\end{definition}

Now we introduce bibundles, invertible bibundles, and locally invertible bibundles. 

\begin{definition}[Bibundle]
\label{def:14}
  Given Lie groupoids $G$ and $H$, a \emph{bibundle} $P:G \to H$ is a manifold $P$ equipped with a left $G$ action along $a$, and a right $H$ action along $a'$, such that
  \begin{itemize}
  \item $P \xar{a}  G_0$ is a principal right $H$ bundle,
  \item $a':P \to H_0$ is $G$-invariant,
  \item the actions of $G$ and $H$ commute.
  \end{itemize}
  A \emph{bibundle morphism} from a bibundle $P$ to a bibundle $Q$ is a smooth map $\alpha:P \to Q$ that commutes with the structure maps and is $G\times H$ equivariant.

    A bibundle $P:G \to H$ is \emph{invertible} or \define{biprincipal} if both $a$ and $a'$ are principal bundles. We say Lie groupoids $G$ and $H$ are \emph{Morita equivalent} if there is an invertible bibundle between them.

  \eod
\end{definition}
If $P:G \to H$ is invertible, then by swapping the $G$ and $H$ actions, we get a bibundle $P^{-1}:H \to G$. The compositions $P^{-1}P$ and $PP^{-1}$ are isomorphic to $\id_G$ and $\id_H$, respectively.
\begin{remark}
  \label{rem:7}
  Every bibundle morphism is a bibundle isomorphism. For details, see \cite[Remark 5.34(5)]{MM2003} or \cite[Remark 3.32]{L2010}.
\eor
\end{remark}
The diagram below depicts a bibundle.
\begin{equation*}
  \begin{tikzcd}
      G \circlearrowright &  P  \ar[dl, "a"'] \ar[dr, "a'"] & \circlearrowleft H \\
      G_0 & &H_0.
  \end{tikzcd}
\end{equation*}

\begin{example}
  \label{ex:5}
  The identity bibundle from a Lie groupoid $G$ to itself, $\id_G:G \to G$, is given by $P = G$, with the left and right actions of $G$ on itself. It is invertible.

  For a more sophisticated example, consider a left Lie group action $G \circlearrowright M$ and a right Lie group action $N \circlearrowleft H$. An bibundle $P: G \ltimes M\to N \rtimes H$ between the action groupoids is given by the diagram
  \begin{equation*}
    \begin{tikzcd}
      & G \circlearrowright P \circlearrowleft H \ar[dl, "a"'] \ar[dr, "a'"]& \\
      G \circlearrowright M & & N \circlearrowleft H,
    \end{tikzcd}
  \end{equation*}
  where $a$ is a principal $H$-bundle, and is $G$-equivariant, and $a'$ is $G$-invariant, and is $H$-equivariant, and the $G$ and $H$ actions on $P$ commute. This bibundle is biprincipal if $a'$ is a principal $G$-bundle.
  \eoe
\end{example}

Bibundles can be composed: for bibundles $P:G \to H$ and $Q:H \to K$, with left and right anchors $a, a'$ and $b, b'$, respectively, the composition $Q \circ P$ is $(P \times_{H_0}Q) /H$ (where the $H$ action on the fibered product $P \fiber{}{H_0} Q$ is $(p,q)\cdot h = (p \cdot h, h^{-1} \cdot q)$), with left and right anchors $\alpha([p, q]) = a(p)$ and $\beta'([p, q]) = b'(q)$. In a diagram,
  \begin{equation*}
    \begin{tikzcd}[sep=small]
      &  & (P \times_{H_0} Q) /H \ar[ddll, "\alpha"', bend right] \ar[ddrr, "\beta'", bend left] &  &  \\
      & P \ar[dl, "a"'] \ar[dr, "a'"] &  &Q \ar[dl, "b"'] \ar[dr, "b'"] &  \\
      G_0 &  & H_0  &  & K_0 
    \end{tikzcd}
  \end{equation*}
See \cite[Remark 3.30]{L2010} for details. However, composition of bibundles is not associative. Instead, it is associative up to bibundle isomorphism, so we have a bicategory:

\begin{definition}[Bicategory of Lie groupoids]
  \label{def:15}
  The bicategory $\bi$ has objects Lie groupoids, arrows bibundles, and 2-arrows morphisms of bibundles.
  \eod
\end{definition}

\begin{remark}
  \label{rem:8}
By identifying isomorphic bibundles, we obtain the Hilsum-Skandalis category $\mbf{HS}$. Its objects are Lie groupoids, and its arrows are isomorphism classes of bibundles.
  \eor
\end{remark}

\begin{example}
  \label{ex:6}
Every smooth functor of Lie groupoids induces a bibundle. Indeed, given a smooth functor $F:G \to H$ with $F_0:G_0 \to H_0$, the principal $H$-bundle $F_0^*(H \xar{t} H_0)$ is naturally a bibundle $G \to H$, which we denote $\langle F\rangle$. The groupoid $G$ acts along the anchor map $\pr_1$, with multiplication $g \cdot (x, h) := (t(g), F(g)h)$.  The map $F \mapsto \langle F \rangle$ respects composition in the sense that $\langle F \circ G \rangle \cong \langle F \rangle \circ \langle G \rangle$. 

A smooth natural transformation between smooth functors $\alpha:F \to G$ gives rise to a bibundle morphism $\langle \alpha \rangle: \langle F \rangle \to \langle G\rangle$. This bibundle morphism is an isomorphism by Remark \ref{rem:7}.
\eoe
\end{example}

We now introduce the notion of the restriction of a bibundle, and then of a locally invertible bibundle.

\begin{definition}
  \label{def:16}
    If $P:G \to H$ is a bibundle as in Definition \ref{def:14}, and $U \subseteq G_0$ and $V \subseteq H_0$ are open subsets such that $a^{-1}(U) \cap (a')^{-1}(V)$, with the actions of $G|_U$ and $H|_V$ described by Remark \ref{rem:3}, is a bibundle, then we call this bibundle the \define{restriction} of $P$ to $G|_U$ and $H|_V$, and denote it by $P|_U^V$.
    \eod
\end{definition}

\begin{lemma}
  \label{lem:3}
  If $G$ and $H$ are Lie groupoids, $P:G \to H$ is a bibundle, and $U \subseteq G_0$ and $V \subseteq H_0$ are open subsets, and $Q:G|_U \to H|_V$ is a bibundle such that the diagram below commutes (up to isomorphism of bibundles)
  \begin{equation}
    \label{eq:2}
    \begin{tikzcd}
      G|_U \ar[r, "Q"] \ar[d, "{\langle \iota_U \rangle}"] & H|_V \ar[d, "{\langle \iota_V \rangle}"] \\
      G \ar[r, "P"] & H,
    \end{tikzcd}
  \end{equation}
  then $P|_U^V$ exists and is isomorphic to $Q$ (here $\iota_U$ is the inclusion functor). Conversely, if $P|_U^V$ exists, then the diagram \eqref{eq:2} commutes with $Q := P|_U^V$.
\end{lemma}

\begin{proof}
  By definition, the bibundle $\langle \iota_U \rangle$ is $U \fiber{\iota_U}{t} G$, which we identify with $t^{-1}(U)$ via $(x,g) \mapsto g$. The $G|_U$ action is by left multiplication. Then
  \begin{equation*}
  P \circ \langle \iota|_U\rangle = (t^{-1}(U) \fiber{s}{a} P)/G  
  \end{equation*}
  We denote its elements by $[g, p]$.  Similarly, we identify the bibundle $\langle \iota_V \rangle$ with $t^{-1}(V)$, and find that
  \begin{equation*}
  \langle \iota_V \rangle \circ Q=  (Q \fiber{b'}{t} t^{-1}(V))/H|_V,
  \end{equation*}
 where $b'$ is the right anchor for $Q$. We denote its elements by $[q, h]$.

  Denote an isomorphism $P \circ \langle \iota_U \rangle \leftrightarrow \langle \iota_V \rangle \circ Q$ element-wise by $[g,p] \leftrightarrow [q,h]$. Then, define a map $Q \to P$ by $q \mapsto g \cdot p$, where $[q, 1_{b'(q)}] \leftrightarrow [g,p]$ under the assumed isomorphism. This is well-defined because $[g_1,p_1] = [g_2,p_2]$ implies $g_1\cdot p_1 = g_2\cdot p_2$. Now, notice that $a(g \cdot p) = t(g) \in U$, and
  \begin{equation*}
  a'(g \cdot p) = a'(p) = b'(q)\in V,   
\end{equation*}
where the last equality follows from the requirement that an isomorphism of bibundles commutes with the anchor maps. So we have defined a map $Q \to P|_U^V$. Its inverse is the map $p \mapsto q \cdot h$, where $[1_{a(p)}, p] \leftrightarrow [q,h]$. Both these maps are smooth, and furthermore they are $G|_U \times H|_V$-equivariant. This shows that $P|_U^V$ is a bibundle, and is isomorphic to $Q$.

  For the converse, we define the diffeomorphism $P \circ \langle \iota_U \rangle \to \langle \iota_V \rangle \circ P|_U^V$ as follows. Given $[g,p]$, denote $p_* := g\cdot p$, and note that $a(p_*) = t(g) \in U$. Now take a local section $\sigma:U \to P|_U^V$ of $a$ about $a(p_*)$, which is possible because $P|_U^V \xar{a} U$ is principal by assumption. Then $q := \sigma(a(p_*))$ and $p_*$ are in the same $a$-fiber, hence, by principality of the $H$ action, there is a unique $h \in H$ (depending smoothly on $p$ and $g$; see Remark \ref{rem:6}) such that $q \cdot h = p_*$. Necessarily $t(h) = a'(q) \in V$. We claim the desired diffeomorphism is $[g,p] \mapsto [q,h]$.
  
We show this is well-defined. Suppose $[g_1,p_1] = [g_2,p_2]$. Then, as above, we have $g_1 \cdot p_1 = g_2 \cdot p_2 =: p_*$, so we only need to show the definition is independent of the section $\sigma$. Suppose $\sigma_i:U \to P|_U^V$ are two local sections of $a$ about $a(p_*)$. Set $q_i := \sigma_i(a(p_*))$, and set $h_i \in H$ to be the unique arrows such that $q_i \cdot h_i = p_*$.  Then both $q_i \cdot h_i$ are in the same $a$-fiber, so there is some unique $\tilde{h} \in H$ such that $q_2 = q_1 \cdot \tilde{h}$.  Then $q_2 \cdot h_2 = q_2 \cdot (\tilde{h}h_2)$, and uniqueness of $h_1$ means $\tilde{h}h_2 = h_1$. It follows that $[q_2,h_2] = [q_1 \cdot \tilde{h}, \tilde{h}^{-1}h_1] = [q_1,h_1]$, as required.

We leave it to the reader to check that this map is $G|_U \times H$-equivariant. It is also smooth: the maps $(g,p) \mapsto (\sigma(a(p_*)), h)$ provide local liftings. Finally, it is a bibundle isomorphism by Remark \ref{rem:7}.
\end{proof}

\begin{definition}[Locally invertible bibundles]
  \label{def:17}
  A bibundle $P:G \to H$ is \emph{locally invertible} if, about every $x \in G_0$ and $y \in a'(a^{-1}(x))$, there are open neighbourhoods $U$ of $x$ and $V$ of $y$, such that the restriction $P|_U^V:G|_U \to H|_V$ is well-defined and is invertible.
  \eod
\end{definition}
 The composition of invertible bibundles is invertible, and the composition of locally invertible bibundles is locally invertible, by Lemma \ref{lem:3}. Our main source of locally invertible bibundles comes from certain smooth functors $F:G \to H$.

\begin{lemma}
  \label{lem:4}
  Suppose $F:G \to H$ is a smooth functor such that, for every $x \in G_0$, there exist open neighbourhoods $U$ of $x$ and $V$ of $F(x)$ such that the restriction $F|^U_V:G|_U \to H|_V$ is an isomorphism. Then $\langle F \rangle$ is a locally invertible bibundle.
\end{lemma}

\begin{proof}
  Take $F:G \to H$ as given in the setup. We will show that $\langle F \rangle|^U_V$ is invertible. First, observe that setting $Q := \langle F|_U^V \rangle$ and $P := \langle F \rangle$ makes diagram \eqref{eq:2} commute, so by Lemma \ref{lem:3}, we see that $\langle F|_U^V \rangle$ is isomorphic to $\langle F \rangle|_U^V$. Therefore it suffices to show $\langle F|^U_V \rangle$ is invertible.

Since $F|_U^V$ is an isomorphism, we simplify notation and assume $F$ is an isomorphism, and prove $\langle F \rangle$ is invertible. In other words, we must show $G_0 \fiber{F}{t} H \xar{s \circ \pr_2} H_0$ is a left $G$-principal bundle. In fact, we show it is isomorphic to a trivial left $G$-bundle. Take the section $\sigma$ of $s \circ \pr_2$ defined by $\sigma(y) := (F^{-1}(y), 1_y)$, and define $\Phi$, in the diagram below
  \begin{equation*}
    \begin{tikzcd}
      G \fiber{s}{\pr_1\sigma} H_0 \ar[r, "\Phi"] \ar[dr, "\pr_2"] & G_0 \fiber{F}{t} H \ar[d, "s \pr_2"] \\
      & H_0
    \end{tikzcd}
  \end{equation*}
 by
  \begin{align*}
    \Phi(g,y) := (t(g), F(g)), \quad \Phi^{-1}(x,h) = (F^{-1}(h), s(h)).
  \end{align*}
Unwinding the definitions yields that these are inverses of each other, and that they are $G$-equivariant, is a matter of unwinding the definitions.
  \end{proof}

  We end this subsection by describing the quotient functor $\mbf{F}:\bi \to \diffeol$. This functor, and parts of the following proposition, have appeared elsewhere in the literature, for instance in \cite[Lemma 6.5]{W2017}. We give the entire argument for completeness; the clause concerning locally invertible bibundles is new.

\begin{proposition}
\label{prop:1}
Suppose $P:G \to H$ is a bibundle. There is a unique map $|P|:|G| \to |H|$ for which the diagram below commutes
  \begin{equation}
    \label{eq:3}
    \begin{tikzcd}[sep=small]
      & P \ar[dl, "a"'] \ar[dr, "a'"] &  \\
      G_0 \ar[d, "\pi"'] &  & H_0 \ar[d, "\pi'"] \\
      {|G|} \ar[rr, "{|P|}"] &  & {|H|}.
    \end{tikzcd}
  \end{equation}
  The map $|P|$ is diffeologically smooth. If $\id_G:G \to G$ is the identity bibundle, then $|\id_G| = \id_{|G|}$. If $Q:H \to K$ is another bibundle, then $|Q \circ P| = |Q| \circ |P|$. If $P, P':G \to H$ are isomorphic, then $|P| = |P'|$. Hence $\mathbf{F}:\bi \to \diffeol$, which takes an object $G \rra G_0$ to $|G|$, takes an arrow $P$ to $|P|$, and takes a 2-arrow $\alpha :P \to Q$ to $1_{|P|}$, is a well-defined functor of bicategories.

  If $P$ is invertible, then $|P|$ is a diffeomorphism. If $P$ is locally invertible, then $|P|$ is a local diffeomorphism.
\end{proposition}

\begin{proof}
  Because the given diffeology on $|G|$, which is the quotient diffeology induced by $\pi$, coincides with the quotient diffeology that is induced by $\pi a$, if $|P|$ is well-defined, then $|P| \pi a = \pi' a'$ implies that $|P|$ is smooth and uniquely determined by $a$ and $a'$: necessarily, $|P|$ is given by
  \begin{equation*}
    [x] \mapsto [a'(p)], \text{ where } a(p) = x.
  \end{equation*}
 We now check this is well-defined. Suppose $g:x \mapsto y$, and $a(p) = x$ and $a(q) = y$. We find an arrow in $H$ taking $a'(q)$ to $a'(p)$. First, $s(g) = x = a(p)$, so $g \cdot p$ is well-defined. Applying $a$ gives
  \begin{equation*}
    a(g\cdot p) = t(g) = y = a(q).  
  \end{equation*}
  Because $P \xar{a} G_0$ is $H$ principal, we can find (unique) arrow $h$ such that $(g \cdot p) \cdot h = q$. The fact the action is well-defined means $t(h) = a'(g\cdot p)$, and this is $a'(p)$ by $G$-invariance of $a'$. On the other hand,
  \begin{equation*}
    a'(q) = a'(g\cdot p \cdot h) = s(h),
  \end{equation*}
  so $h$ is the desired arrow.

  Now, suppose $P:G \to H$ and $Q:H \to K$ are bibundles, with left and right anchors $a, a'$ and $b, b'$, respectively. Recall that composition $Q \circ P$ is given by the diagram
  \begin{equation*}
    \begin{tikzcd}[sep=small]
      &  & P \times_{H_0} Q /H \ar[ddll, "\alpha"', bend right] \ar[ddrr, "\beta'", bend left] &  &  \\
      & P \ar[dl, "a"'] \ar[dr, "a'"] &  &Q \ar[dl, "b"'] \ar[dr, "b'"] &  \\
      G_0 \ar[d] &  & H_0 \ar[d] &  & K_0 \ar[d]  \\
      {|G|} \ar[rr, "{|P|}"] \ar[rrrr, bend right, "{|Q \circ P|}", dashed] &  & {|H|} \ar[rr, "{|Q|}"] &  &  {|K|}
    \end{tikzcd}
  \end{equation*}
  Choose $x\in G_0$, then choose $p \in P$ such that $a(p) = x$, then choose $q$ such that $b(q) = a'(p)$. Then
  \begin{equation*}
    |Q| \circ |P|([x]) = Q([a'(p)]) = [b'(q)] 
  \end{equation*}
  But we can now pick $[p,q] \in Q \circ P$, where $\alpha([p,q]) = x$ and $\beta'([p,q]) = b'(q)$. Then $|Q \circ P|([x]) = [b'(q)]$ too, as required.

  Suppose $\alpha:P \to Q$ is an isomorphism of bibundles. Fix $a(p) =x$, so that $|P|([x]) = [a'(p)]$. Then $b(\alpha(p)) = x$, so $|Q|([x]) = [b'(\alpha(p))]$. But $b'\alpha = a'$, so $|Q|([x]) = [a'(p)] = |P|([x])$. 

  If $P$ is invertible, the inverse of $|P|$ is $|P^{-1}|$. If $P$ is locally invertible, for $x \in G_0$ and $y \in a'(a^{-1}(x))$, take neighbourhoods $U$ of $x$ and $V$ of $y$ such that $P|_U^V$ is invertible. By Lemma \ref{lem:3} we have the diagram \eqref{eq:2}, and passing to the quotient yields the diagram
  \begin{equation*}
    \begin{tikzcd}
      {|U|} \ar[d] \ar[r, "{|P|_U^V|}"] & {|V|} \ar[d] \\
      {|G|} \ar[r, "{|P|}"] & {|H|}.
    \end{tikzcd}
  \end{equation*}
 See Lemma \ref{lem:5} for the identification of $|U| = U/G|_U$ with an open subset of $|G|$. The vertical arrows are inclusions, and therefore the top arrow is $|P|$ restricted to a map $|U| \to |V|$. But the top arrow is also a diffeomorphism because $P|_U^V$ is invertible, so we conclude $|P|$ is a local diffeomorphism. 
\end{proof}

\subsection{Quasifold groupoids}
\label{sec:quasifold-groupoids}
Now we introduce quasifold groupoids. Our definition has not previously appeared in the literature.

\begin{definition}
\label{def:18}
  A $n$-\emph{quasifold groupoid} is a Lie groupoid $G \rra G_0$, with Hausdorff arrow space, such that: for each $x \in G_0$, there is an open neighbourhood $U$ of $x$, a countable group $\Gamma$ acting affinely on $\R^n$, an open subset $\msf{V}$ of $\R^n$, and an isomorphism of Lie groupoids $F: G|_U \to (\Gamma \ltimes \R^n)|_{\msf{V}}$. We call a collection $\cl A = \{F:G|_U \to (\Gamma \ltimes \R^n)|_{\msf{V}}\}$ of such Lie groupoid isomorphisms such that the $U$ cover $G_0$ a \emph{(quasifold) atlas} for $G$. We call a groupoid of the form $(\Gamma \ltimes \R^n)|_{\msf{V}}$ a \emph{model quasifold groupoid}.
  \eod
\end{definition}

\begin{remark}
  \label{rmk:2}
  In this definition, we do not assume $\Gamma$ acts effectively on $\R^n$ (i.e.\ that the action homomorphism $\Gamma \to \Aff(\R^n)$ is injective), nor do we assume that $\msf{V}$ is $\Gamma$-invariant.
\eor

\end{remark}

  \begin{remark}
    \label{rem:9}
    Like with diffeological quasifolds, our notion of quasifold groupoid is \emph{local}, meaning that, for a quasifold groupoid $G$, for each point $x$, and each neighbourhood $U'$ of $x$, there is an open neighbourhood $U$ of $x$ contained in $U'$ such that $G|_U$ is isomorphic to a model quasifold groupoid $(\Gamma \ltimes \R^n)|_{\msf{V}}$. We do not know if, in general, these models can always be chosen with $\msf{V}$ being $\Gamma$-invariant. We do know this is the case if all the groups $\Gamma$ are finite (meaning the quasifold is an orbifold, c.f.\ Example \ref{ex:8}).
    \eor
  \end{remark}

\begin{definition}
  \label{def:19}
The category $\qfoldgrpd$ is the sub-bicategory of $\bi$ whose objects are quasifold groupoids. Taking only the locally invertible bibundles as arrows, we get the bicategory $\qfoldgrpd^{\text{loc-iso}}$.
  \eod
\end{definition}

\begin{example}
  \label{ex:7}
  In the setting of Example \ref{ex:1}, the action groupoids $(\Z + \alpha \Z) \ltimes \R$ for irrational tori are quasifold groupoids. Combining our main Theorem \ref{thm:1} with Iglesias and Donato's result in \cite{DIZ1985}, we recover the fact that two such action groupoids are Morita equivalent if and only if $\alpha$ and $\beta$ are related by a fractional linear transformation with integer coefficients.
  \eoe
\end{example}

\begin{example}
  \label{ex:8}
  Taking the groups $\Gamma$ to be finite subgroups of $\GL(\R^n)$, we call the resulting quasifold groupoid an \emph{orbifold} groupoid. Kozsul's slice theorem \cite{K1953} implies that an \'{e}tale Lie groupoid is an orbifold groupoid if and only if it is locally proper.\footnote{A Lie groupoid is \define{locally proper} if about every $x \in G_0$, there is a neighbourhood $U$ such that $G|_U$ is proper.} For the argument, see \cite[Proposition 5.30]{MM2003}. Many authors represent orbifolds by \'{e}tale, proper Lie groupoids. In contrast, our orbifold groupoids need not be globally proper. The weighted non-singular branched manifolds of \cite{McDuff-2006-book}, which come from Kuranishi atlases, provide examples of locally proper Lie groupoids that are not \emph{a priori} proper.
\eoe
\end{example}

To show the quotient functor $\mbf{F}:\bi \to \diffeol$ restricts to a functor $\qfoldgrpd \to \diffeolqfold$, we require the following technical lemma. 

\begin{lemma}
  \label{lem:5}
 Fix a Lie groupoid $G \rra G_0$ and an open subset $U \subseteq G_0$. In the following diagram, the bottom inclusion map is a diffeomorphism with its image, $\pi(U)$, where $\pi(U)$ carries the subset diffeology induced from $G_0/G$.
    \begin{equation*}
    \begin{tikzcd}
      (U, \text{ subset diffeol.})  \ar[d, "\pi_U"] \ar[r, hook]  & G_0 \ar[d, "\pi"]
     \\
      (U/G|_U, \text{ quotient diffeol.}) \ar[r, hook] & (G_0/G, \text{ quotient diffeol.}).
    \end{tikzcd}
  \end{equation*}
  In particular, if $\Gamma$ is a Lie group acting smoothly on a manifold $M$, and $\pi:M \to M/\Gamma$ is the quotient map, and $V$ is an open subset of $M$, then the quotient diffeology on $\pi(V)$ induced from $V$ coincides with the subset diffeology induced from $M/\Gamma$.
\end{lemma}

For an open subset $U \subseteq G_0$, the above lemma justifies using the notation $|U|$ to refer unambiguously to the diffeological spaces $U/G|_U = \pi(U)$.

\begin{proof}
  % The inclusion and its inverse are
  % \begin{align*}
  %   \bm{x'} &\mapsto \pi(x'), \text{ where } x' \in \bm{x'} \\
  %   \bm{x} &\mapsto \pi_U(x), \text{ where } x \in U \cap \bm{x}.
  % \end{align*}
  % To show these are smooth, it suffices to prove the downward arrows are subductions (the diffeological version of a submersion, see \cite[Article 1.46]{IZ2013}). For $\pi_U$, this is due to the definition of the quotient diffeology on $U/G|_U$. We now prove this for $\pi$. Take a plot $p:\msf{U} \to \pi(U)$, which is a plot $p:\msf{U} \to G/G_0$ with image in $\pi(U)$. For $r \in \msf{U}$, this means we may take an open neighbourhood $\msf{V}$ of $r$ in $\msf{U}$, and a lift $q:\msf{V} \to G_0$ such that $\pi q = p|_{\msf{V}}$. Since $\pi(q(r)) \in \pi(U)$, choose $x \in U$ and $g \in G$ so that $g:q(r) \to x$.

  Fix a map $p:\msf{U} \to \pi(U)$. Suppose that $p$ is a plot for the quotient diffeology. Then it locally lifts to smooth maps to $U$, hence to $G_0$. So it is a plot of $G_0/G$, hence of $\pi(U)$ as a subset of $G_0/G$. Now suppose that $p$ is smooth with respect to the subset diffeology. Then it locally lifts to smooth maps to $G_0$. Fix $r \in \msf{U}$. Let $\msf{V}$ be an open neighbourhood $r$ in $\msf{U}$, and let $q:\msf{V} \to G_0$ be a smooth map such that $\pi \circ q = p|_{\msf{V}}$. Choose $x \in U$ such that $p(r) = \pi(x)$. Then $\pi(q(r)) = \pi(x)$. Choose an arrow $g: q(r) \to x$.

  Because $s$ is a submersion, we may take a section $\sigma$ of $s$ about $q(r)$ such that $\sigma(x) = g$. Because $t$ is continuous and $t\sigma(x) \in U$, we may restrict the domain of $\sigma$ so that $t\sigma$ always lands in $U$. Then $t\sigma q$, on a sufficiently small domain, is the required local lift of $p$. 
\end{proof}

\begin{proposition}
\label{prop:2}
  The orbit space of an $n$-quasifold groupoid $G$ is a diffeological $n$-quasifold: if $\cl A = \{F:G|_U \to (\Gamma \ltimes \R^n)|_{\msf{V}}\}$ is an atlas for $G$, the induced maps $|\cl{A}| = \{|F|: |U| \to \msf{V}/\Gamma\}$ form an atlas for $|G|$.
\end{proposition}

\begin{proof}
  Fix an $n$-quasifold groupoid $G$, and $x \in G_0$.  Let $F:G|_U \to (\Gamma \ltimes \R^n)|_{\msf{V}}$ be an element of $\cl{A}$ such that $U$ contains $x$. Since $F$ is a Lie groupoid isomorphism, by Proposition \ref{prop:1} it descends to a diffeological diffeomorphism $|F|: |U| \to |(\Gamma\ltimes \R^n)|_{\msf{V}}|$. By Lemma \ref{lem:5}, we find $|F|$ is a diffeomorphism from an open subset of $X$ about $[x]$, to an open subset of $\R^n/\Gamma$. Because $G_0$ is second-countable, so is $|G|$, so by Definition \ref{def:8}, we conclude $|G|$ is a diffeological $n$-quasifold.
\end{proof}

\subsection{Effective quasifold groupoids}

Every quasifold groupoid is \'{e}tale, because the groups $\Gamma$ in the local models $(\Gamma \ltimes \R^n)|_{\msf{V}}$ are discrete. Now we define an effective \'{e}tale groupoid. We will use pseudogroups. In this subsection $P$ denotes a pseudogroup, and not a bibundle.

\begin{definition}
\label{def:20}
 A \emph{pseudogroup} $P$ on a manifold $M$ is a set of transitions $M \to M$ (see Definition \ref{def:6}) that contains the identity, is closed under composition, inversion, and restriction to open subsets, and satisfies the following locality axiom:
  \begin{quote}
 If $f$ is a transition on $M$ and there exists a cover of its domain for which $f|_U \in P$ for every element $U$ of the cover, then $f \in P$.
\end{quote}
\eod
\end{definition}

A pseudogroup $P$ on $M$ partitions $M$ into \emph{orbits} $\{f(x) \mid f \in P\}$. The intersection of pseudogroups is a pseudogroup. For a set $P_0$ of transitions, we call the intersection of all pseudogroups containing $P_0$ the pseudogroup \emph{generated} by $P_0$.

\begin{remark}
\label{lem:6}
  A pseudogroup $P$ on $M$ is generated by a set of transitions $P_0$ if and only if every $\phi \in P$ is locally a composition of elements of $P_0$ and their inverses. Precisely, for every $x \in \dom \phi$, there are $\phi_i \in P_0$, and $\epsilon_i \in \{1,-1\}$, such that $\germ_x\phi = \germ_x \phi_1^{\epsilon_1} \circ \cdots \circ \phi_n^{\epsilon_n}$.
\end{remark}

  \begin{example}
    \label{ex:9}\
    \begin{itemize}
    \item The collection of all transitions from a manifold $M$ to itself is a pseudogroup, called the \define{Haefliger pseudogroup}.
    \item Given an \'{e}tale Lie groupoid $G\rra G_0$, the collection of \emph{local bisections} of $G_0$, defined as $\{t \circ \sigma \mid \sigma \text{ is a local inverse of } s\}$, generates a pseudogroup on $G_0$, which we denote $\Psi(G)$.
    \end{itemize}
    \eoe
  \end{example}

  \begin{remark}
  \label{prop:3}
  Conversely, given a pseudogroup $P$ on $M$, we form its \emph{germ groupoid} $\Gamma(P) \rra M$. Its arrows are the germs of elements of $P$. The source map is $s(\germ_x\phi) = x$, the target is $t(\germ_x\phi) = \phi(x)$, the multiplication is $\germ_x\phi \cdot \germ_{x'} \phi' = \germ_{x'} \phi\phi'$, the unit map is $x \mapsto \germ_x\id$, and the inversion is $(\germ_x\phi)^{-1} = \germ_{\phi(x)}\phi^{-1}$. We equip the arrow space $\Gamma(P)$ with the smooth structure given by the following atlas. For each $\phi \in P$ and chart $r:U \to \Omega$ of $M$ with $U \subseteq \dom \phi$, take the chart
  \begin{equation*}
\{\germ_x\phi \mid \phi(x) \in U\} \to \Omega, \quad \text{given by} \quad \germ_x\phi \mapsto r(x).
 \end{equation*}
  \eod
  \end{remark}

  Now we introduce the effect functor, and the definition of an effective \'{e}tale groupoid.
  
  \begin{definition}
    \label{lem:7}
    Let $G$ be \'{e}tale Lie groupoid. Let $\Psi(G)$ be its associated pseudogroup (see Example \ref{ex:9}). Let $\Gamma(\Psi(G))$ be the its germ groupoid. The \define{effect} functor $\Eff:G \to \Gamma(\Psi(G))$ is identity on $G_0$, and
    \begin{equation*}
      \Eff(g) := \germ_{s(g)} (t \circ \sigma),
    \end{equation*}
    where $\sigma$ a local inverse of $s$ taking $s(g)$ to $g$. This is a surjective local diffeomorphism on arrows. An \'{e}tale Lie groupoid $G$ is \define{effective} if $\Eff$ is also injective on arrows. In this case, $G$ is isomorphic to $\Eff(G) = \Gamma(\Psi(G))$.
    \eod
  \end{definition}
  The germ groupoid $\Gamma(P)$ of a pseudogroup is effective. In general, for an \'{e}tale Lie groupoid $G$ and pseudogroup $P$, we have the relations
  \begin{equation}\label{eq:4}
    \Gamma(\Psi(G)) = \Eff(G), \quad \Psi(\Gamma(P)) = P.
  \end{equation}

  We define $\qfoldgrpd^{\text{loc-iso}}_{\text{eff}}$ as the subcategory of $\qfoldgrpd^{\text{loc-iso}}$ consisting of effective quasifold groupoids.

  \begin{corollary}
    \label{cor:2}
    The quotient functor $\mbf{F}$ restricts to a functor $\mbf{F}_{\text{Quas}}$ from $\qfoldgrpd^{\text{loc-iso}}_{\text{eff}}$ to $\diffeolqfold^{\text{loc-iso}}$.
  \end{corollary}

  \begin{proof}
    This is Proposition \ref{prop:2} combined with Proposition \ref{prop:1}.
  \end{proof}

  \ifdebug \newpage \fi
  
\section{From Diffeological Quasifolds to Quasifold Groupoids}
\label{sec:from-diff-quas}

In this section, we show that the quotient functor $\mbf{F}_{\text{Quas}}$ is essentially surjective. Namely, to a diffeological quasifold $X$, equipped with atlas an $\cl{A}$, we associate an effective quasifold groupoid $\Gamma(\cl{A})$ whose quotient $\mbf{F}(\Gamma(\cl{A}))$ is diffeomorphic to $X$. As the notation suggests, $\Gamma(\cl{A})$ will be a germ groupoid of a pseudogroup.

\begin{definition}
  \label{def:21}
  Let $\cl{A} = \{F_i:U_i \to \msf{V}_i/\Gamma_i\}$ be a countable atlas of a diffeological quasifold $X$. Let
  \begin{equation*}
     \pi:\bigsqcup \msf{V}_i \to \bigsqcup \msf{V}_i/\Gamma_i, \quad \text{and} \quad F^{-1}:\bigsqcup \msf{V}_i/\Gamma_i \to X
   \end{equation*}
   be the maps induced by $\pi_i$ and $F_i^{-1}$.  Denote
   \begin{equation*}
     \Pi := F^{-1}\pi.
   \end{equation*}
   Let $\Psi(\cl{A})$ be the pseudogroup of transitions $\phi$ on $\bigsqcup \msf{V}_i$ such that $\Pi \phi = \Pi$. Let $\Gamma(\cl{A})$ be the germ groupoid associated to the pseudogroup $\Psi(\cl{A})$; see Remark \ref{prop:3}. We have
  \begin{equation*}
    \Gamma(\cl{A}) = \bigsqcup_{i,j} \{\germ_x\phi \mid \phi \in \diffloc^{\cl{A}}(\msf{V}_i,\msf{V}_j), \ x \in \dom \phi\},
  \end{equation*}
  where $\diffloc^{\cl{A}}(\msf{V}_i,\msf{V}_j)$ is the set of transitions $\phi:\msf{V}_i \dashrightarrow \msf{V}_j$ in $\Psi(\cl{A})$.
  \eod
\end{definition}

\begin{remark}
\label{rem:10}
For each $i$, the transitions $\phi\in \diffloc^{\cl{A}}(V_i,V_i)$ are precisely the transitions of $\msf{V}_i$ which preserve $\Gamma_i$-orbits. By Corollary \ref{cor:1}, for each $x \in \dom \phi$, there is some $\gamma \in \Gamma_i$ such that $\germ_x\phi = \germ_x\gamma$. Therefore $\Gamma(\cl{A})|_{\msf{V}_i} = \{\germ_x \gamma \mid x \in \msf{V}_i, \ \gamma \in \Gamma_i\}$, and $\msf{V}_i/(\Gamma(\cl{A})|_{\msf{V}_i}) = \msf{V}_i/\Gamma_i$.
\eor
\end{remark}

The groupoid $\Gamma(\cl{A})$ appears in \cite{IZL2018} and \cite{IZP2020} as a diffeological groupoid. We show $\Gamma(\cl{A})$ is an effective quasifold groupoid; in particular, it is Lie and Hausdorff.

\begin{lemma}
  \label{lem:8}
  For a diffeological quasifold $X$ equipped with countable atlas $\cl A$, the Lie groupoid $\Gamma(\cl{A})$ is an effective quasifold groupoid.
\end{lemma}

\begin{proof}
  Since $\Gamma(\cl{A})$ is a germ groupoid, it is effective. We show $\Gamma(\cl{A})$ has an atlas of quasifold charts. Consider a chart $F:U \to \msf{V}/\Gamma$ in $\cl{A}$ with $\msf{V}  \subseteq \R^n$ and $\Gamma$ is a subgroup of $\Aff(\R^n)$.

  By Remark \ref{rem:10}, elements of $\Gamma(\cl{A})|_{\msf{V}}$ are precisely the germs of elements of $\Gamma$. We may then consider the surjective map
  \begin{equation}\label{eq:5}
     \Gamma \times \msf{V} \to \Gamma(\cl{A})|_{\msf{V}}, \quad (\gamma, x) \mapsto \germ_x\gamma.
  \end{equation}
  This is also injective, because for affine transformations, if $\gamma \neq \gamma'$, then $\germ_x \gamma \neq \germ_x \gamma'$.

  For fixed $\gamma$, the induced map $\msf{V} \to \Gamma(\cl{A})|_{\msf{V}}$ is simply $\germ \gamma$, which is a smooth diffeomorphism. Since $\Gamma$ is discrete, we conclude the map in \eqref{eq:5} is a local diffeomorphism; because it is bijective, it is a diffeomorphism. Together with the identity on the base, the map \eqref{eq:5} gives the desired functor from the model quasifold.

  Now we show $\Gamma(\cl{A})$ is Hausdorff. Take any two distinct elements of $\Gamma(\cl{A})$. Write them as $\germ_x\phi$ and $\germ_{x'} \phi'$, for $\phi:\msf{V}_i \dashrightarrow \msf{V}_j$ and $\phi':\msf{V}_{i'} \dashrightarrow \msf{V}_{j'}$. If $i \neq i'$ or $j \neq j'$, then $\phi$ and $\phi'$ have disjoint domains or codomains. The corresponding subsets of $\Gamma(\cl{A})$ are the desired disjoint neighbourhoods. Now suppose that $i = i'$ and $j = j'$. There are three cases.

  \begin{itemize}
  \item[Case 1] If $x \neq x'$, separate $x$ and $x'$ by disjoint neighbourhoods $\msf{U}$ and $\msf{U}'$. The required neighbourhoods are the subsets of $\Gamma(\cl{A})$ corresponding to $\phi|_{\msf{U}}$ and $\phi'|_{\msf{U'}}$.
  \item[Case 2]  If $x = x'$ but $\phi(x) \neq \phi'(x)$, take neighbourhoods $\msf{U}$ and $\msf{U}'$ of $x$ such that $\phi(\msf{U})$ and $\phi'(\msf{U}')$ are disjoint. The required neighbourhoods are the subsets of $\Gamma(\cl{A})$ corresponding to $\phi|_{\msf{U}}$ and $\phi'|_{\msf{U'}}$.
  
    \item[Case 3] If $x = x'$ and $\phi(x) = \phi'(x)$, then $\germ_x(\phi^{-1}\phi')$ is an element of $\Gamma(\cl{A})|_{\msf{V}_i}$, which by Remark \ref{rem:10}, we can identify with some $\gamma \in \Gamma_i$. Because $\germ_x\phi \neq \germ_x\phi'$, this $\gamma$ is not the identity. Therefore $\germ_y\phi \neq \germ_y\phi'$ for all $y$ in some small neighbourhood $\msf{U}$ of $x$.  The required neighbourhoods are the subsets of $\Gamma(\cl{A})$ corresponding to $\phi|_{\msf{U}}$ and $\phi'|_{\msf{U}}$.
    \end{itemize}
\end{proof}

Thus we have associated to $X$ the effective quasifold groupoid $\Gamma(\cl{A})$. Now we show that the functor $\mbf{F}$ respects this assignment.
\begin{proposition}
\label{prop:4}
  Fix a quasifold $X$ with countable atlas $\cl{A}$, as in Definition \ref{def:21}. In the notation from Definition \ref{def:21}, the map $\Pi$ descends to a diffeological diffeomorphism $|\Pi|: \bigsqcup \msf{V}_i / \Gamma(\cl{A}) \to X$. In particular, $\mbf{F}(\Gamma(\cl{A})) \cong X$. 
\end{proposition}

\begin{proof}
The orbits of $\Gamma(\cl{A})$ are the equivalence classes of the relation: $x \sim \phi(x)$, for some $i$ and $j$ and $\phi \in \diffloc^{\cl{A}}(\msf{V}_i,\msf{V}_j)$. Therefore $\Pi$ descends to the quotient. By results from diffeology, $|\Pi|$ is diffeologically smooth. Because $\Pi$ is onto, it descends to a surjective map. It remains to show injectivity of $|\Pi|$ and smoothness of its inverse.
\begin{itemize}
\item $|\Pi|$ is injective:

  Suppose $\Pi(x) = \Pi(x')$, where $x \in \msf{V}_i$ and $x' \in \msf{V}_j$. This means $F^{-1}_i\pi_i(x) = F^{-1}_j\pi_j(x')$. Applying $F_j$ on both sides, and denoting $F_{ij} := F_j F_i^{-1}$, we get
  \begin{equation*}
    F_{ij} \pi_i(x) = \pi_j(x'). 
  \end{equation*}
  The map $F_{ij}$ is a transition $\msf{V}_i/\Gamma_i \dashrightarrow \msf{V}_j/\Gamma_j$. By Lemma \ref{lem:2}, we can lift it to a transition $\phi_{ij}: \msf{V}_i \dashrightarrow \msf{V}_j$ sending $x$ to $x'$. Therefore $x$ and $x'$ are in the same orbit, hence the induced map $|\Pi|$ is injective.
\item $|\Pi|$ has an inverse:

  The inclusion $\iota: \msf{V}_i \hookrightarrow \bigsqcup \msf{V}_i$ is smooth, and if $x$ and $x'$ are in the same $\Gamma_i$ orbit, then $\iota(x)$ and $\iota(x')$ are in the same $\Gamma(\cl{A})$ orbit. Therefore the inclusion descends to a smooth map $\ol \iota: \msf{V}_i/\Gamma_i \to \bigsqcup \msf{V}_i/\Psi_0$. The composition $\iota F_i:U_i \to \iota(\msf{V}_i/\Gamma_i)$ is smooth, and it is the inverse of the restriction $|\Pi|: \iota(\msf{V}_i/\Gamma_i) \to U_i$. This shows $|\Pi|$ is a local diffeomorphism, and since it is bijective, it is a diffeomorphism.
\end{itemize} 
\end{proof}

\ifdebug \newpage \fi

\section{Lifting Local Diffeomorphisms Between Diffeological Quasifolds }
\label{sec:lift-local-isom}
In this section we show that the quotient functor $\mbf{F}_{\text{Quas}}$ is surjective on arrows. In other words, given effective quasifold groupoids $G$ and $H$, and a local diffeomorphism $f:|G| \to |H|$, we construct a locally invertible bibundle $P:G \to H$ such that $|P| = f$. First, we replace the effective quasifolds $G$ and $H$ with convenient germ groupoids.

\begin{definition}
  \label{def:22}
  For an \'{e}tale Lie groupoid $G$, define the pseudogroup $\diffloc^G(G_0)$ to consist of all transitions $\phi$ of $G_0$ that preserve $G$-orbits. Denote its germ groupoid by
  \begin{equation*}
  \Gamma^G := \Gamma(\diffloc^G(G_0)).  
  \end{equation*}
   \eod
\end{definition}
Note that the pseudogroup $\Psi(G)$ of local bisections of $G_0$ (see Example \ref{ex:9}) is a subset of $\diffloc^G(G_0)$.
\begin{lemma}
\label{lem:9}
If $G$ is a quasifold groupoid, then $\Psi(G) = \diffloc^G(G_0)$, and hence $|G| = |\Gamma^G|$. 
\end{lemma}

\begin{proof}
Fix $\phi \in \diffloc^G(G_0)$. It suffices to show that about every $x \in \dom \phi$, there is a neighbourhood on which $\phi$ restricts to an element of $\diffloc^G(G_0)$.
  \begin{itemize}
  \item[Case 1:] $\phi(x) = x$. Choose a quasifold groupoid chart $F:G|_U \to (\Gamma \ltimes \R^n)|_{\msf{V}}$ about $x$. The (partially defined) map $F_0\phi F_0^{-1}$ restricts to a transition $V \dashrightarrow V$ about $F_0(x)$ that preserves $\Gamma$-orbits. By Corollary \ref{cor:1}, we may choose $\gamma \in \Gamma$ such that $\germ_{F_0(x)} F_0\phi F_0^{-1} = \germ_{F_0(x)}\gamma$. Consider the partially defined map
  \begin{equation*}
    \sigma: U \dashrightarrow G|_U, \quad x' \mapsto F^{-1}(\gamma, F_0(x')). 
  \end{equation*}
  This is a section of $s$, and composing with $t$ yields,
  \begin{equation*}
    tF^{-1}(\gamma, F_0(x')) =  F_0^{-1}(\gamma \cdot F_0(x')) = \phi(x').
  \end{equation*}
  Therefore, near $x$, we have $\phi = t\sigma \in \Psi(G)$.
  \item[Case 2:] $\phi(x) \neq x$. Since preserves $G$-orbits, there exists an arrow $g \in G$ such that $g:x \to \phi(x)$. Let $\sigma$ be a section of $s$ through $g$. Then $(t\sigma)^{-1} \phi$ is an element of $\diffloc^G(G_0)$ fixing $x$. By Case 1, this is in $\Psi(G)$. Therefore, near $x$, we have $\phi = (t\sigma)((t\sigma)^{-1}\phi) \in \Psi(G)$.
  \end{itemize} 
 
\end{proof}
\begin{corollary}
  By the correspondence \eqref{eq:4}, we see $\Eff(G) = \Gamma^G$. If $G$ is also effective, then $G \cong \Eff(G) = \Gamma^G$.
\end{corollary}

Now we show that the quotient functor $\mbf{F}_{\text{Quas}}$ is surjective on arrows.

\begin{proposition}
  \label{prop:5}
  Suppose $G$ and $H$ are effective quasifold groupoids. Assume $f:|G| \to |H|$ is a local diffeomorphism. Then there is a locally invertible bibundle $P:G \to H$ such that $|P| = f$. If $f$ is a diffeomorphism, then $P$ is invertible.
\end{proposition}

\begin{proof}
  Let
  \begin{equation*}
    \pi:G_0 \to |G| \quad \text{and} \quad \pi':H_0 \to |H|
  \end{equation*}
  be the quotient maps.  Because $f$ is a local diffeomorphism, $f(|G|)$ is an open subset of $|H|$. The groupoid $H|_{(\pi')^{-1}(f(|G|))}$ is an open subgroupoid of $H$. The bibundle associated to the inclusion functor is a locally invertible bibundle $H|_{(\pi')^{-1}(f(|G|))} \to H$, which lifts the inclusion $f(|G|) \hookrightarrow |H|$. By functoriality of $\mbf{F}$, it suffices to find a locally invertible bibundle $Q:G \to H|_{(\pi')^{-1}(f(|G|))}$. Therefore, without loss of generality, we assume $f$ is surjective.

   Fix quasifold groupoid atlases of $G$ and $H$,
  \begin{equation*}
    \cl{A} := \{F_i: G|_{U_i} \to (\Gamma_i \ltimes \R^n)|_{\msf{V}_i}\}, \quad \cl{A}' := \{F_\alpha':H|_{U_\alpha'} \to (\Gamma_\alpha' \ltimes \R^{n'})|_{\msf{V}_\alpha'}\}.
  \end{equation*}
  We have open covering maps
  \begin{equation*}
   \iota:\cl{U} := \bigsqcup U_i \to G_0, \quad \iota':\cl{U}' := \bigsqcup U_\alpha' \to H_0.
  \end{equation*}
  We can pull back $G$ by $\iota$ to get a Lie groupoid $\fk{G} \rra \cl{U}$. Its arrows from  $x \in U_i$ to $y \in U_j$ are, by Lemma \ref{lem:9}, germs of $G$-orbit preserving transitions $\phi:U_i \dashrightarrow U_j$ such that $\phi(x) = y$. In other words, we can identify $\fk{G}$ with
  \begin{equation*}
    \Gamma (\{\phi \in \diffloc(\cl{U}) \mid \pi \phi = \pi \text{ on } \dom \phi\}).
  \end{equation*}
  Here we interpret $\pi \phi$ to mean $\pi \iota \phi$. In particular, we view $\fk{G}$ as a subgroupoid of the Haefliger groupoid $\Gamma(\cl{U})$ of germs of transitions on $\cl{U}$. We have a similar description of $\fk{H} := (\iota')^*H$.

Define
  \begin{equation*}
    Q := \{\germ_y \psi \mid \psi \text{ is a transition }U_\alpha' \dashrightarrow U_i \text{ and }f \pi \psi = \pi'\}.
  \end{equation*}
  As an open subset of the arrow space of the Haefliger groupoid $\Gamma(\cl{U}' \sqcup \cl{U})$, this is a manifold. We get submersions $s,t$ from $Q$, induced by the source and target of the ambient Haefliger groupoid. Using Proposition \ref{prop:2} and the given atlases to identify $\pi(U_i)$ with $\msf{V}_i/\Gamma_i$, and $\pi(U_\alpha')$ with $\msf{V}_\alpha'/\Gamma_\alpha'$, we can use Lemma \ref{lem:2} to see that, for any $x \in U_i$ and $y \in U_j'$ for which $f(\pi(x)) = \pi(y)$, there is a transition $\ti{f}:U_i \dashrightarrow U_\alpha'$ that is a local lift of $f$ taking $x$ to $y$. Then $\germ_y(\ti{f})^{-1} \in Q$ has source $y$ and target $x$. Because we assumed $f$ is surjective, we conclude $s:Q \to \cl{U}'$ and $t:Q \to \cl{U}$ are surjective.

  The groupoid $\fk{G}$ acts on $Q$ from the left along the anchor map $t$, by composition of germs. The groupoid $\fk{H}$ acts on $Q$ from the right along the anchor map $s$, also by composition. These actions are smooth, and commute because composition is associative. The maps $t$ and $s$ are $\fk{H}$ and $\fk{G}$ invariant, respectively. In a diagram,
  \begin{equation*}
    \begin{tikzcd}
      \fk{G} \circlearrowright &  Q  \ar[dl, "t"'] \ar[dr, "s"] & \circlearrowleft \fk{H} \\
      \cl{U} & &\cl{U}'.
    \end{tikzcd}
  \end{equation*}

We show $Q$ is a right principal $\fk{H}$-bundle, and hence a bibundle. For $x_0 \in U_i$, because $t$ is surjective we can find some $\germ_{\psi_0^{-1}(x_0)}\psi_0 \in Q$. Say $\psi_0$ is a transition $U_\alpha' \dashrightarrow U_i$, and denote $W := \dom \psi_0^{-1} \subseteq U_i$. Take the local section $\sigma$ of $t$ defined by
\begin{equation*}
  \sigma:W \to Q, \quad \sigma(x) := \germ_{\psi_0^{-1}(x)} \psi_0.
\end{equation*}
We define the map $\Phi$ (not to be confused with $\Phi$ from Section \ref{sec:from-diff-quas}) fitting in the diagram below
\begin{equation*}
  \begin{tikzcd}
    t^{-1}(W) \ar[r, "\Phi"] \ar[d, "t"] & W' \fiber{s\sigma}{t} \fk{H} \ar[dl, "\pr_1"] \\
    W
  \end{tikzcd}
\end{equation*}
by
\begin{equation*}
  \Phi(\germ_y \psi) := (\psi(y), \germ_y \psi_0^{-1} \psi).
\end{equation*}
Because $\germ_y\psi \in t^{-1}(W)$, necessarily $\psi(y) \in W$, so we can compose $\psi_0^{-1}$ and $\psi$, and $\germ_y \psi_0^{-1} \psi$ is well-defined. Furthermore, first using that $\psi_0 \in Q$ and then $\psi \in Q$,
\begin{equation*}
  \pi'\psi_0^{-1}\psi = f\pi\psi = \pi' \text{ on } \dom \psi.
\end{equation*}
Therefore $\Phi$ is well-defined. Unravelling the definitions yields that $\Phi^{-1}$ exists, that both maps are smooth, and that $\Phi$ is $\fk{H}$-equivariant.

Finally, we show $Q$ is locally invertible. For $y_0 \in U_\alpha'$, because $s$ is surjective, we can find some $\germ_{y_0} \psi_0 \in Q$. Say $\psi_0$ is a transition $U_\alpha' \dashrightarrow U_i$. Choose an open neighbourhood $W'$ of $y_0$ in $U_\alpha'$, and an open neighbourhood $W$ of $\psi_0(y_0)$ in $U_i$, such that
\begin{align*}
  &\psi_0 \text{ restricts to a transition } W' \to W,\quad \text{and} \\
  &f \text{ restricts to a diffeomorphism } f|_{|W'|}: |W'| \to |W|.
\end{align*}
We denote $Q|^{W'}_W := s^{-1}(W') \cap t^{-1}(W)$. Because of our choice of $W'$ and $W$, we can rewrite the condition that $\psi \in Q|^{W'}_W$ as
\begin{equation}\label{eq:6}
  \pi \psi = f|_{|W'|}\pi' \text{ on } \dom \psi
\end{equation}
Take the local section $\sigma$ of $s$ defined by
\begin{equation*}
  \sigma: W' \to Q|^{W'}_W, \quad \sigma(y) := \germ_y \psi_0.
\end{equation*}
Define the map $\Phi$ fitting in the diagram below
\begin{equation*}
  \begin{tikzcd}
    \fk{G} \fiber{s}{t\sigma} W' \ar[dr, "\pr_2"'] & s^{-1}(W') \ar[l, "\Phi"'] \ar[d, "s"] \\
    & W'
  \end{tikzcd}
\end{equation*}
by
\begin{equation*}
  \Phi(\germ_y\psi) := (\germ_{\psi_0(y)}\psi \psi_0^{-1}, y).
\end{equation*}
Because $\psi \in s^{-1}(W')$, and $\psi_0$ is a transition $W' \to W$, we can compose $\psi$ and $\psi_0^{-1}$, hence $\germ_{\psi_0(y)}\psi\psi_0^{-1}$ is well-defined. Furthermore, first using condition \eqref{eq:6} and $\psi \in Q|^{W'}_W$, and then using $\psi_0 \in Q|^{W'}_W$,
\begin{equation*}
  \pi\psi\psi_0^{-1} = f|_{\pi(W')}\pi'\psi_0^{-1} = f|_{\pi(W')}f\pi = \pi \text{ on } \dom \psi_0^{-1}.
\end{equation*}
Therefore $\Phi$ is well-defined. Checking the other required properties hold is a matter of unravelling definitions. We conclude that $Q$ is a locally-invertible bibundle.

If $f$ is a diffeomorphism, then in the argument above we may choose $W' := \dom \psi_0$, and there is no need to choose $W$ or restrict $f$ or $Q$. This shows that $Q$ is invertible.

The bibundle $Q$ fits into the diagram
\begin{equation*}
  \begin{tikzcd}
    \fk{G} \ar[r, "\langle \iota \rangle" ] \ar[rrr, bend left, "Q"] \ar[d] & G\ar[d] & H\ar[d] & \ar[l, "\langle \iota' \rangle"'] \ar[d] \fk{H} \\
    {|\fk{G}|} \ar[r] & {|G|} \ar[r, "f"] & {|H|} \ar[r,leftarrow] & {|\fk{H}|}.
  \end{tikzcd}
\end{equation*}

Both $\langle \iota \rangle$ and $\langle \iota' \rangle$ are invertible bibundles. Therefore the locally invertible bibundle $P:= (\langle \iota' \rangle \circ Q) \circ \langle \iota \rangle^{-1}$ is the desired lift of $f$. If $f$ is a diffeomorphism, then $Q$, hence $P$, is invertible.
\end{proof}

\ifdebug \newpage \fi

\section{Local Diffeomorphisms Have Essentially Unique Lifts}
\label{sec:an-equiv-categ}

In this section, we show that the quotient functor $\mbf{F}_{\text{Quas}}$ is injective on arrows, up to 2-isomorphism. In other words, for two locally invertible bibundles $P,Q:G \to H$ between effective quasifold groupoids that descend to the same map $|P| = |Q|:|G| \to |H|$, we give an isomorphism of bibundles $P \cong Q$. We begin with the case of functors.

\begin{lemma}
  \label{lem:10}
  Suppose $G, H$ are effective quasifold groupoids, and $F,K:G \to H$ are smooth functors that induce the same map $|F|=|K|:|G| \to |H|$. If $|F|$ (hence $|K|$) is a local diffeomorphism, then there is a smooth natural transformation $F \to K$.
\end{lemma}
\begin{proof}
  Because $F,K:G_0 \to H_0$ both lift the local diffeomorphism $|F| = |K|$, they are also local diffeomorphisms; this is a consequence of Lemma \ref{lem:2}. Therefore, we may cover $G_0$ with open sets $U_i$, such that
  \begin{align*}
    F_i &:= F|_{U_i}:U_i \to F(U_i)=: V_i' \text{ and }\\
    K_i &:= K|_{U_i}:U_i \to K(U_i) =: V_i''
  \end{align*}
  are diffeomorphisms. The map $KF_i^{-1}:V_i' \to H_0$ preserves $H$-orbits: using the lifting properties of $F$ and $K$, we write
  \begin{align*}
    KF_i^{-1}(x') &\xar{\pi'} |K|([F_i^{-1}(x')]), \\
    x' = F(F_i^{-1}(x')) &\xar{\pi'} |F|([F_i^{-1}(x')]),
  \end{align*}
  and the right sides are equal because $|K| = |F|$. Therefore $\germ_{x'} KF_i^{-1} \in \Gamma^H$, and we have the smooth map
  \begin{equation*}
    \bigsqcup U_i \to \Gamma^H, \quad x \mapsto \germ_{F(x)} KF_i^{-1}, \text{ where } x \in U_i.
  \end{equation*}
  Because $H$ is an effective quasifold groupoid, by Lemma \ref{lem:9} the effect map is an isomorphism of groupoids $\Eff:H \to \Gamma^H$. So the map $x \mapsto \Eff^{-1}(\germ_{F(x)}KF_i^{-1})$ is a smooth map $\bigsqcup U_i \to H$. Denote this map by $\alpha$. We claim $\alpha$ is a smooth natural transformation $\alpha:\iota^*F \to \iota^* K$, where $\iota:\bigsqcup U_i \to G_0$ is the covering map. Because $\langle\iota \rangle$ is invertible, we conclude there is a smooth natural transformation from $F$ to $G$.

  \begin{itemize}
  \item The source of $\alpha(x)$ is $F(x)$. Its target is $KF_i^{-1}(F(x))$, and $x \in U_i$, so this is $K(x)$.
  \item Fix an arrow $g:x \to y$ in $G$, and say $x \in U_i$ and $y \in U_j$. Fix a section $\sigma$ of $s$ over a neighbourhood of $x$ in $U_i$, such that $\sigma(x) = g$. Shrink the domain of $\sigma$ so that $t\sigma \in U_j$ always holds.

    Then, $F\sigma F_i^{-1}$, restricted to some small neighbourhood of $F(x)$ in $V_i'$, is a well-defined section of $s$ such that $F\sigma F_i^{-1}(F(x)) = F(g)$. We may then compute
    \begin{align*}
      \Eff(\alpha(y) F(g)) &= \germ_{F(y)}KF_j^{-1} \cdot \germ_{F(x)} t(F\sigma F_i^{-1}) &&\text{by definition}\\
                           &= \germ_{F(x)} KF_j^{-1}tF\sigma f_i^{-1} &&\text{multiplying}\\
                           &= \germ_{F(x)} KF_j^{-1} Ft\sigma f_i^{-1} &&\text{because } F \text{ is a smooth functor}.
    \end{align*}
    By choice of the domain of $\sigma$, the first part $t\sigma F_i^{-1}$ has image in $U_j$. Therefore $F_j^{-1}F(t\sigma F_i^{-1}) = t\sigma F_i^{-1}$, and because $k$ is a smooth functor,
    \begin{equation*}
      \Eff(\alpha(y) F(g)) = \germ_{F(x)} Kt\sigma F_i^{-1} = \germ_{F(x)} tK\sigma F_i^{-1}.
    \end{equation*}
On the other hand, $K\sigma K_i^{-1}$, restricted to some small domain of $K(x)$ in $U_i''$, is a well-defined section of $s$ such that $K\sigma K_i^{-1}(K(x)) = K(g)$. So then
    \begin{align*}
      \Eff(K(g)\alpha(x)) &= \germ_{K(x)} t(K\sigma K_i^{-1}) \cdot \germ_{F(x)} KF_i^{-1} \\
                          &= \germ_{F(x)} tK\sigma K_i^{-1}KF_i^{-1} \\
                          &= \germ_{F(x)} tK\sigma F_i^{-1} &\text{because } F_i^{-1} \text{ lands in } U_i.
    \end{align*}
    Therefore $\Eff(\alpha(y)F(g))$ and $\Eff(K(g)\alpha(x))$ are the same arrow. Because $H$ is effective, we conclude $\alpha(y)F(g) = K(g)\alpha(x)$, as required.
  \end{itemize}
\end{proof}

Now we handle the general case.

\begin{proposition}
  \label{prop:6}
  Suppose $P,Q:G \to H$ are bibundles between effective quasifold groupoids that descend to the same map $|P| = |Q|:|G| \to |H|$, and that $|P|$ is a local diffeomorphism. Then $P$ and $Q$ are isomorphic bibundles. 
\end{proposition}

\begin{proof}
  By \cite[Lemma 3.37]{L2010}, there is a cover $\iota:\bigsqcup U_i \to G_0$ and smooth functors $F,K:\iota^*G \to H$ such that
  \begin{align*}
    P \circ \langle \iota \rangle &\text{ is isomorphic to } \langle F\rangle, \text{ and } \\
    Q \circ \langle \iota \rangle &\text{ is isomorphic to } \langle K\rangle.
  \end{align*}
 Since $\langle \iota \rangle$ is invertible, it suffices to prove $\langle F\rangle \cong \langle K \rangle$. But $\iota^*G$ is also an effective quasifold groupoid, and $F,K$ are two smooth functors that induce the same map of quotient spaces, namely $|P| \circ |\langle\iota \rangle| = |Q| \circ |\langle\iota\rangle|$. So by Lemma \ref{lem:10} and Example \ref{ex:6}, we find $\langle F\rangle \cong \langle K \rangle$.
\end{proof}
This proposition proves that the quotient functor $\mbf{F}_{\text{Quas}}$ on effective quasifold groupoids is injective on arrows, up to 2-isomorphism. We gather our results in the main theorem.
\begin{theorem}
  \label{thm:1}
  Let $\qfoldgrpd^{\text{loc-iso}}_{\text{eff}}$ be the bicategory whose objects are effective quasifold groupoids, whose arrows are locally invertible bibundles, and whose 2-arrows are morphisms of bibundles. Let $\diffeolqfold^{\text{loc-iso}}$ be the bicategory whose objects are diffeological quasifolds, whose arrows are local diffeomorphisms, and whose 2-arrows are trivial. Then the quotient functor $\mathbf{F}_{\text{Quas}}: \qfoldgrpd^{\text{loc-iso}}_{\text{eff}} \to \diffeolqfold^{\text{loc-iso}}$, which is well defined by Corollary \ref{cor:2}, is essentially surjective, is surjective on arrows, and injective on arrows up to 2-isomorphism.
\end{theorem}

\begin{proof}\

  \begin{itemize}
  \item It is essentially surjective: given a diffeological quasifold $X$, take an atlas $\cl{A}$. Then $\Gamma(\cl{A})$ is an effective quasifold groupoid with orbit space diffeomorphic to $X$ by Proposition \ref{prop:4}.
  \item It is surjective on arrows: given a local diffeomorphism $f:|G| \to |H|$, where $G$ and $H$ are effective quasifold groupoids, we get a locally invertible bibundle $P:G \to H$ such that $|P| = f$ from Proposition \ref{prop:5}.
  \item It injective on arrows up to 2-isomorphism: if $P,Q:G \to H$ are locally invertible bibundles between effective quasifold groupoids that induce the same map on the orbit spaces, then $P \cong Q$ by Proposition \ref{prop:6}.

  \end{itemize}
\end{proof}

Moving to the Hilsum-Skandalis category $\mbf{HS}$ (see Remark \ref{rem:8}) lets us phrase this result more succinctly.

\begin{corollary}
  \label{cor:3}
The quotient functor, from the category whose objects are effective quasifold groupoids, and whose arrows are isomorphism classes of locally invertible bibundles, to the category whose objects are diffeological quasifolds, and whose arrows are local diffeomorphisms, is an equivalence of categories.
\end{corollary}

Denote by $\qfoldgrpd^{[\text{loc-iso}]}_{\text{eff}}$ the subcategory of $\mbf{HS}$ induced by $\qfoldgrpd^{\text{loc-iso}}_{\text{eff}}$ . Corollary \ref{cor:3} resolves the question of gluing morphisms in $\qfoldgrpd^{[\text{loc-iso}]}_{\text{eff}}$. This is a question because in $\mbf{HS}$, morphisms generally do not glue, i.e. if $P$ and $Q$ are bibundles between Lie groupoids $G$ and $H$, and there is an open cover $\{U_i\}$ of $G_0$ such that $P|_{U_i} \cong Q|_{U_i}$ for all $i$, it does not follow that $P \cong Q$. For examples, see \cite[Lemma 3.41]{L2010} or \cite[page 2491]{HM2004}. However, for quasifolds:

\begin{corollary}
  \label{cor:4}
  If $P$ and $Q$ are locally invertible bibundles between effective quasifold groupoids $G$ and $H$, and there is an open cover $\{U_i\}$ of $G_0$ such that $P|_{U_i} \cong Q|_{U_i}$ for all $i$, then $P \cong Q$.
\end{corollary}
\begin{proof}
  For each $i$, denote the inclusion $U_i \to G_0$ by $\iota_{U_i}$. By Lemma \ref{lem:3} and Proposition \ref{prop:1},
  \begin{equation*}
    |P|_{U_i}| = |P \circ \langle \iota_{U_i}\rangle| = |P| \circ |\langle \iota_{U_i} \rangle| = |P||_{|U_i|}.
  \end{equation*}
The assumption $P|_{U_i} \cong Q|_{U_i}$ implies that $|P||_{|U_i|} = |Q||_{|U_i|}$. Since the $|U_i|$ cover $|G_0|$, and since diffeologically smooth maps are local, this means $|P| = |Q|$. By Corollary \ref{cor:3}, we conclude that $P \cong Q$.
\end{proof}

\ifdebug \newpage \fi

\section{A non-example}
\label{sec:a-non-example} 
By Theorem \ref{thm:1}, given two quasifold groupoids $G$ and $H$, a diffeomorphism $|G| \to |H|$ lifts to a Morita equivalence of $G$ and $H$. In this last section, we describe a non-example where this lifting property fails. Fix a smooth function $h:\R \to \R$ that is flat at $0$ and positive everywhere else, and such that the vector field
\begin{equation*}
  \xi := h\frac{\partial}{\partial x}
\end{equation*}
is complete. For example, we can take
\begin{equation*}
h(x) :=
\begin{cases}
  e^{-\frac{1}{x^2}} & \text{if } x \neq 0 \\
  0 &\text{if } x=0.
\end{cases}
\end{equation*}
The time-one flow of $\xi$ is a diffeomorphism $\psi:\R \to \R$ such that
\begin{itemize}
\item its jet at $0$ coincides with the jet at $0$ of the identity map, and
\item $\psi(\R_{>0}) = \R_{>0}$ and $\psi(\R_{<0}) = \R_{<0}$.
\end{itemize}
The inverse $\psi^{-1}:\R \to \R$ has the same properties. So does
\begin{equation*}
  \hat{\psi}(x) :=
  \begin{cases}
    \psi(x) &\text{if } x \geq 0 \\
    \psi^{-1}(x) &\text{if } x < 0.
  \end{cases}
\end{equation*}
Let $G$ be a copy of $\Z$, with action on $\R$ generated by $\psi$. Let $H$ be a copy of $\Z$, with action on $\R$ generated by $\hat{\psi}$. Then $G$ and $H$ have the same orbits, so
\begin{equation*}
  \R/G = \R/H =: \R/{\sim}.
\end{equation*}

  \begin{proposition}
    \label{prop:7}
  The identity map on $\R/{\sim}$ does not lift to a biprincipal bibundle between the action groupoids $G \ltimes \R$ and $H \ltimes \R$. Moreover, these action groupoids are not Morita equivalent.
\end{proposition}

\begin{proof}
   For a contradiction, suppose these groupoids are Morita equivalent. By Example \ref{ex:5}, an invertible bibundle $P: G \ltimes \R\to \R \rtimes H$ (we can view the $H$ action as a right action because $H$ is abelian) is given by the diagram
  \begin{equation*}
    \begin{tikzcd}
      & G \circlearrowright P \circlearrowleft H \ar[dl, "a"'] \ar[dr, "a'"]& \\
      G \circlearrowright \R & & \R \circlearrowleft H,
    \end{tikzcd}
  \end{equation*}
  where $a$ is a principal $H$-bundle, and is $G$-equivariant, and $a'$ is a principal $G$-bundle, and is $H$-equivariant, and the $G$ and $H$ actions on $P$ commute. Because they are bundles over $\R$, both $a$ and $a'$ are trivial. Choose a global section $\sigma$ of $a$. By Proposition \ref{prop:1},
  \begin{equation*}
  \varphi := a'\sigma:\R \to \R
\end{equation*}
descends to a diffeomorphism $\R/G \to \R/H$. By equivariance of $a$, both $\sigma(1 \cdot x)$ and $1 \cdot \sigma(x)$ are in the same $a$-fiber (note that $1$ is the generator of the $G$ and $H$ action). Therefore we have the smooth map
  \begin{equation*}
    \R \to P \times_{a} P , \quad x \mapsto (\sigma(1 \cdot x), 1 \cdot \sigma(x)).
  \end{equation*}
  By principality of $P \xar{a} \R$, this yields a smooth map $\eta:\R \to H$ given by
  \begin{equation*}
    \sigma(1 \cdot x) = (1 \cdot \sigma(x)) \cdot \eta(x),
  \end{equation*}
  and because $H$ is discrete, $\eta \in H$ is constant. Therefore
  \begin{equation*}
    \varphi(1 \cdot x) = \varphi(x) \cdot \eta, \text{ and so } \varphi(k \cdot x) = \varphi(x) \cdot (k\eta).
  \end{equation*}
  There are two cases.
  \begin{itemize}
  \item If $\eta = 0$, then $\varphi$ is $G$-invariant. By definition of the $G$ action, for each $x$, we have $k \cdot x \to 0$ as $k \to -\infty$, therefore $\varphi(k \cdot x) = \varphi(x) \to \varphi(0)$, and thus $\varphi(x) = \varphi(0)$ for all $x$. But this contradicts the fact $\varphi$ descends to a diffeomorphism $\R/G \to \R/H$.
  \item If $\eta \neq 0$, then we get a contradiction as follows: for $x$ such that $\operatorname{sign}(\eta) = -\operatorname{sign}(\varphi(x))$, 
    \begin{equation*}
      \varphi(k \cdot x) = \hat{\psi}^{k\eta}(\varphi(x)) = \psi^{\operatorname{sign}(\varphi(x))k\eta}(\varphi(x)) \to \pm \infty \text{ as } k \to -\infty
    \end{equation*}
    On the other hand, $\varphi(k\cdot x) \to \varphi(0) = 0$ as $k \to -\infty$, so we have our contradiction.
  \end{itemize}
\end{proof}

Furthermore,

\begin{proposition}
  \label{prop:8}
  $\R/{\sim}$ is not a diffeological quasifold.
\end{proposition}
  \begin{proof}
  Seeking a contradiction, suppose $\R/{\sim}$ is a diffeological quasifold. Take an open neighbourhood $U$ of $[0]$ in $X$, a countable subgroup $\Gamma$ of $\Aff(\R)$, a $\Gamma$-invariant open subset $\msf{V}$ of $\R$, and a diffeomorphism $F:U \to \msf{V}/\Gamma$. Fix some $x_0 \in \msf{V}$ such that $F([0]) = [x_0]$.

  By the exact same argument from the first half of the proof of Lemma \ref{lem:2}, we may find locally defined lifts $f,s:\R \to \R$ of $F$ and $F^{-1}$, respectively, such that $f(0) = x_0$ and $s(x_0) = 0$. Then $fs$ is a locally defined map that preserves $\Gamma$ orbits, and fixes $x_0$.

The composition $f\psi s$ is then also a locally defined map that preserves $\Gamma$ orbits and fixes $x_0$. Fix an open interval $\msf{W}$ about $x_0$ such that both $fs$ and $f\psi s$ are defined on $\msf{W}$.

  By Lemma \ref{lem:1}, there are $\gamma, \delta \in \Gamma$ such that $fs = \gamma$ and $f\psi s = \delta$ on $\msf{W}$. Therefore $(fs)' = \gamma'$, and in particular $s'$ is non-vanishing on $\msf{W}$. It follows that $s$ is a diffeomorphism on a neighbourhood of $x_0$. Similarly, $f'$ does not vanish on $s(\msf{W})$, and so $f$ is a diffeomorphism on a neighbourhood of $0$.

Since $\psi'(0) = 1$,
  \begin{equation*}
    \gamma' = (fs)'(x_0) = (f \psi s)'(x_0) = \delta'.
  \end{equation*}
  Thus $\gamma$ and $\delta$ differ only by a translation; because $\psi(0) = 0$, we see that $\gamma = \delta$. In other words, $fs = f \psi s$ on $\msf{W}$. Since $f$ and $s$ are diffeomorphisms, we conclude that $\psi = \id$ in a neighbourhood of $0$, which contradictions the definition of $\psi$. Therefore no such diffeomorphism $F$ exists, and $\R/{\sim}$ is not a diffeological 1-quasifold.
\end{proof}

We suspect that the key feature of the $G$ action on $\R$ allowing for Propositions \ref{prop:7} and \ref{prop:8} is that the $G$-action is not \emph{jet-determined}; all the $\psi^k$ have the same jet at $0$, yet are different diffeomorphisms for different $k$. This is what allows for the existence of the smooth function $\hat{\psi}$. Conversely, given a jet-determined action of a discrete group $\Gamma$ on $\R^n$, at the time of writing we do not know if a lemma analogous to Lemma \ref{lem:1} holds. 

The $G$ action on $\R$ also gives rise to a counter-example in foliation theory. We will assume familiarity with foliations and their holonomy.

\begin{proposition}
  \label{prop:9}
  There exists a foliated manifold $(M, \cl{F})$ for which the quotient space $M/\cl{F}$ is not a diffeological quasifold.
\end{proposition}
\begin{proof}
Consider $\R^2$ with coordinates $(t,x)$, equipped with the foliation $\cl{F}$ spanned by $\frac{\partial}{\partial t}$. The leaves of $\cl{F}$ are the horizontal lines $\R \times \{x\}$. We have a $\Z$ action on $\R^2$ given by
\begin{equation*}
  k \cdot (t,x) := (t+k, \psi^k(x)),
\end{equation*}
where $\psi$ is the flow from the beginning of this section. This action is free, properly discontinuous, and preserves $\cl{F}$. Therefore $M := \R^2/\Z$ is a manifold, the quotient $\pi:\R^2 \to M$ is a covering map, and $\cl{F}/\Z := d\pi(\cl{F})$ is a foliation on $M$. The leaves of $\cl{F}/\Z$ are of the form $\pi\left(\R \times \{\psi^k(x)\} \mid k \in \Z\right)$, for $x \in \R$. A connected total transversal to $\cl{F}/\Z$ is given by $T := \pi(\{0\} \times \R) \cong \R$.

The holonomy pseudogroup on $T$ associated to $\cl{F}/\Z$ is the pseudogroup generated by the group $\{\psi^k\}$. In particular, the quotient $M/(\cl{F}/\Z)$ is diffeologically diffeomorphic to $T/\{\psi^k\}$.\footnote{This is because $T/\{\psi^k\}$ is the orbit space of the pullback of the holonomy groupoid associated to $\cl{F}/\Z$ to the transversal $T$, and these groupoids are Morita equivalent.}  But as we have seen in Proposition \ref{prop:7}, $T/\{\psi^k\}$ is not a diffeological quasifold. Therefore neither is $M/(\cl{F}/\Z)$.
\end{proof}

Proposition \ref{prop:9} raises the following question: for which foliations $\cl{F}$ is the leaf space $M/\cl{F}$ a diffeological quasifold? Some candidates include the null foliation on a compact, connected presymplectic toric manifold, in the sense of Ratiu and Zung \cite{RZ2019} or Lin and Sjamaar \cite{LS2019}, or \emph{Riemannian} foliations, as defined in \cite{MM2003}. We leave sorting through these candidates to another paper.

{\ifdebug \newpage \else \end{document} \fi}

%%%%%%%%%%%%%%%%%%%%%%%%%%%%%%%%%%%%%%%%%%%%%%%%%%%%%%%%%%%%%%%%
\section{Extras}
\label{sec:extras}

%%%%%%%%%%%%%%%%%%%%%%%%%%%%%%%%%%%%%%%%%%%%%%%%%%%%%%%%%%%%%%%% 
\smallskip\noindent\textbf{Potential Journals}
%%%%%%%%%%%%%%%%%%%%%%%%%%%%%%%%%%%%%%%%%%%%%%%%%%%%%%%%%%%%%%%%
\begin{itemize}
\item Indagationes
\item IMRN
\item Annalen
\item Transformation groups
\item Compositio
\item crelle journal
\item Illinois journal of mathematics
\item Pacific journal of mathematics
\item Journal of geometry and physics
\end{itemize}

We chose Selecta Mathematics. MCS is 22A22 and 58H05

\newpage
%%%%%%%%%%%%%%%%%%%%%%%%%%%%%%%%%%%%%%%%%%%%%%%%%%%%%%%%%%%%%%%% 
\smallskip\noindent\textbf{Other definition of principal bibundle}
%%%%%%%%%%%%%%%%%%%%%%%%%%%%%%%%%%%%%%%%%%%%%%%%%%%%%%%%%%%%%%%%

\begin{lemma}
\label{lem:11}
  Given a principal right $H$-bundle $P \xar{\pi} B$, the map $\pi$ is an $H$-invariant surjective submersion, and the map
  \begin{equation}
\label{eq:7}
    P \tensor[_a]{\times}{_t} H \to P \times_\pi P, \quad (p, h) \mapsto (p, p\cdot h)
  \end{equation}
  is a diffeomorphism. Conversely, given a right $H$ action on $P$ and a $H$-invariant surjective submersion $P \xar{\pi} B$ such that \eqref{eq:7} is diffeomorphism, the data $P \xar{\pi} B$ with the $H$ action is a principal right $H$ bundle.
\end{lemma}

\begin{proof}
  For the first direction, the existence of the local model \eqref{eq:1} shows that $\pi$ is an $H$-invariant surjective submersion. As for the map \eqref{eq:7}, it is injective because the $H$ action is free, and surjective because $H$ acts transitively on the $\pi$-fibers. Therefore its inverse is well-defined. To show it is smooth, we give its local expression and observe this is smooth. Fix $(p',, q')$ such that $\pi(p') = b' = \pi(q')$. Choose $U \ni b'$, $\sigma$, and $\Phi$ from the definition of principal bundle, such that we have \eqref{eq:1}. Denote $\Phi(p) = (\pi(p), \phi(q))$. Consider the map
  \begin{equation*}
    P_U \times_\pi P_U \to P_U \tensor[_a]{\times}{_t} H, \quad (p, q) \mapsto (p, \phi(p)^{-1}\phi(q)).
  \end{equation*}
  If this is well-defined, then it is smooth and the inverse of \eqref{eq:7} about $(p', q')$. First, by definition of the target of $\Phi$, we have $a\sigma\pi(p) = t(\phi(p))$. So if $(p, q) \in P \times_\pi P$, then $t(\phi(p)) = t(\phi(q))$, and thus the composition on our proposed map makes sense. Now, since $\Phi$ is $H$-equivariant, we have
  \begin{equation*}
    (\pi(p), \phi(p)) = \Phi(p) = \Phi(p \cdot 1_{a(p)}) = (\pi(p), \phi(p) 1_{a(p)}).
  \end{equation*}
  In particular, $s(\phi(p)) = a(p)$. So $a(p) = t(\phi(p)^{-1}\phi(q))$, so our proposed map is indeed well-defined.

  Conversely, assuming \eqref{eq:7} is a diffeomorphism, it is immediate that $H$ acts freely and its orbits are the $\pi$-fibers. Now we must give local trivializations. Because $\pi$ is a surjective submersion, for fixed $b \in B$, there is an open neighbourhood $U$ and section $\sigma:U \to P$ of $\pi$. For $p \in P_U$, denote by $h_p$ the arrow in $H$ such that $p \cdot h_p = \sigma(\pi(p))$. The assignment $p \mapsto h_p$ is well-defined and smooth because we extract $h_p$ from the diffeomorphism \eqref{eq:7}. Then, consider 
  \begin{align*}
    \Phi: P_U \to U \tensor[_{a\sigma}]{\times}{_t} H, &\quad p \mapsto (\pi(p), h_p^{-1}) \\
    U \tensor[_{a\sigma}]{\times}{_t} H \to P_U, &\quad (b, h) \mapsto \sigma(b) \cdot h
  \end{align*}
  These are both well-defined and smooth, and are inverses of each other. So $\Phi$ is a diffeomorphism. It is equivariant because $h_{p\cdot h} = h^{-1}h_p$, so $\Phi(p \cdot h) = (\pi(p), h_p^{-1}h) = \Phi(p) \cdot h$.
\end{proof}

\newpage
%%%%%%%%%%%%%%%%%%%%%%%%%%%%%%%%%%%%%%%%%%%%%%%%%%%%%%%%%%%%%%%%
\smallskip\noindent\textbf{Morphisms of principal bibundles}
%%%%%%%%%%%%%%%%%%%%%%%%%%%%%%%%%%%%%%%%%%%%%%%%%%%%%%%%%%%%%%%%

We also have the notion of morphisms between principal bundles.

\begin{definition}
  \label{def:23}
Given two principal $H$ bundles $P \xar{\pi} B$ and $Q \xar{\pi'} B$, over the same base, we call a smooth map $\alpha:P \to Q$ a \emph{bundle morphism} if $\pi' \alpha = \pi$, and $\alpha$ is $H$-equivariant.
\end{definition}

\begin{remark}
\label{rem:11}
  A bundle morphism $\alpha:P \to Q$ between principal $H$ bundles $P$ and $Q$ is necessarily a diffeomorphism. See, for example, \cite{MM2003}, Remark 5.34 (5).
\end{remark}
\begin{proof}
  Pick $p_0 \in P$ and denote $q_0:= \alpha(p_0)$. Say $H$ acts on $P$ with anchor $a$, and on $Q$ with anchor $b$. By principality, take a neighbourhood $U$ of $\pi(p_0) = \pi'(q_0)$, sections $\sigma$ of $\pi$ and $\tau$ of $\pi'$, and bundle morphisms $\Phi:P_U \to (a\sigma)^*(H \xar{t} H_0)$ and $\Psi:Q_U \to (a\sigma)^*(H \xar{t} H_0)$. By equivariance, $\alpha$ restricts to a map $P_U \to Q_U$. Denote $\ti{\alpha} := \Psi \alpha \Phi^{-1}$.

  Because $\ti{\alpha}$ is a bundle map, it preserves the projections, so it takes $(x, 1_{a\sigma(x)}) \in U \tensor[_{a\sigma}]{\times}{_t} H$ to some element $(x, f(x)) \in U \tensor[_{b\tau}]{\times}{_t} H$. The map $f:U \to H$ is smooth, and uniquely determines $\ti{\alpha}$, because by equivariance
  \begin{equation*}
    \ti{\alpha}(x,h) = \ti{\alpha}(x, 1_{a\sigma(x)}) \cdot h = (x, f(x) \cdot h).
  \end{equation*}
  Then $(y,h') \mapsto (y, f(y)^{-1}\cdot h')$ is the smooth local inverse of $\alpha$. This shows $\alpha$ is a local diffeomorphism.

  To see $\alpha$ is injective, suppose $\alpha(p) = \alpha(p')$. Then $\pi(p) = \pi(p')$, so $p = p' \cdot h$ for some unique $h$. But then $\alpha(p) = \alpha(p') \cdot h$, and the action is free so $h = 1_{a(p')}$ and $p = p'$. To see $\alpha$ is surjective, suppose $q \in Q$. Then $\pi'(q) = \pi(p)$ for some $p$, and $\alpha(p)$ and $q$ are in the same $\pi'$ fiber. Since the $\pi'$ fibers are $H$-orbits, there is some $h$ such that $\alpha(p \cdot h) = \alpha(p) \cdot h = q$
\end{proof}

\newpage
%%%%%%%%%%%%%%%%%%%%%%%%%%%%%%%%%%%%%%%%%%%%%%%%%%%%%%%%%%%%%%%%
\smallskip\noindent\textbf{Pullbacks of bibundles}
%%%%%%%%%%%%%%%%%%%%%%%%%%%%%%%%%%%%%%%%%%%%%%%%%%%%%%%%%%%%%%%%

We can use the pullback of principal bundles to pullback bibundles.

  \begin{definition}[Pullback bibundle]
  \label{def:24}
  The \define{pullback} of a bibundle $P:G \to H$ along a smooth submersion $\phi:M \to G_0$ is a bibundle $\phi^*P:\phi^*G \to H$, where the manifold $\phi^*P$ and and principal right $H$ action are given by the pullback $\phi^*(P\xar{a} G_0)$, and the anchor for the left $\phi^*G$ action is $\pr_1$, and the multiplication is $(m \xar{\fk{g}} n) \cdot (m, p) := (n, g \cdot p)$, where $\phi(m) \xar{g} \phi(n)$ is the arrow in $G$ corresponding to $\fk{g}$.
  \eod
\end{definition}

In particular, if $U \subseteq G_0$ is an open subset, and $\iota:G|_U \to G$ is the inclusion functor (where $G|_U$ has arrows $s^{-1}(U) \cap t^{-1}(U)$), we may pull back a bibundle $P:G \to H$ to a bibundle $P_U:G|_U \to H$. This lets us define locally invertible bibundles.

\newpage
%%%%%%%%%%%%%%%%%%%%%%%%%%%%%%%%%%%%%%%%%%%%%%%%%%%%%%%%%%%%%%%%
\smallskip\noindent\textbf{Old proof of Proposition \ref{prop:5}. It is much more confusing and perhaps wrong.}
%%%%%%%%%%%%%%%%%%%%%%%%%%%%%%%%%%%%%%%%%%%%%%%%%%%%%%%%%%%%%%%%

  We use the notation from the previous section, specifically Definition \ref{def:21}. We proceed in three steps.
\begin{enumerater}
\item Given a local diffeomorphism of quasifold groupoids $f:X \to X'$, and atlases $\cl{A}$ and $\cl{A}'$ of $X$ and $X'$, we construct a locally invertible bibundle $Q:\Gamma(\cl{A}) \to \Gamma^{X'}(\cl{A}')$ that lifts $\ol{\Phi'} \circ f \circ (\ol{\Phi})^{-1}$.
\item Given an effective quasifold groupoid $G$, form a covering groupoid $\fk{G}$ (that depends on a quasifold atlas) that is Morita equivalent to $G$, and show this is isomorphic to $\Gamma^{G/G_0}(\cl{A})$. This isomorphism descends to the identity on orbit spaces.
\item Assemble the bibundles from the previous two steps into $P$.
\end{enumerater}

\begin{proposition}
\label{prop:10}
Suppose $X$ and $X'$ are diffeological quasifolds with countable atlases $\cl{A}$ and $\cl{A}'$, and $f:X \to X'$ is a local diffeomorphism. We may construct an invertible bibundle $Q: \Gamma(\cl{A}) \to \Gamma^{X'}(\cl{A}')$, such that $f = (\ol{\Phi'})^{-1} \circ \ol{Q} \circ \ol{\Phi}$, in the notation of Lemma \ref{lem:6}. Moreover, if $f$ is a diffeomorphism, then $Q$ is invertible. In that case, $\Gamma(\cl{A})$ and $\Gamma^{X'}(\cl{A}')$ are Morita equivalent.
\end{proposition}

\begin{proof}
  Fix atlases of $X$ and $X'$,
  \begin{equation*}
    \cl{A} = \{F_i:U_i \to V_i/\Gamma_i\}\text{ and }\cl{A}' = \{F_j':U_j' \to V_j'/\Gamma_j'\}.
  \end{equation*}
  By assumption on $f$, we can choose a countable cover $W_\alpha$ of $X$ such that the restrictions $f_\alpha:W_\alpha \to f(W_\alpha)$ are diffeomorphisms.  For each $j,\alpha$, choose a $\Gamma_j$-invariant open subset $V_{j\alpha}' \subseteq V_j'$ such that $V_{j\alpha}'/\Gamma_j = F_j'(U_j \cap f(W_\alpha))$. We can choose such $V_{j\alpha}'$ because the right hand set is open in $V_j/\Gamma_j$. The \define{pullback} atlas
  \begin{equation*}
    f^*\cl{A} := \left\{F_j' \circ f_\alpha : f_\alpha^{-1}(U_j') \to V_{j\alpha}'/\Gamma_j'\right\}
  \end{equation*}
  is a countable quasifold atlas of $X$. The following diagram summarizes our situation.
  \begin{equation*}
    \begin{tikzcd}
      V_i \ar[d,"\pi_i"] \ar[dr, "\Phi_i"] & & &  &  & V_{j\alpha}' \ar[d, "\pi_j'"] \ar[dl, "\Phi_j'"'] \ar[dll, bend right, "f_\alpha^{-1} \Phi_j'"'] \\
      V_i/\Gamma_i &  \ar[l, "F_i", "\cong"'] U_i \ar[r, "\subseteq"'] & X &  \ar[l, "\supseteq"] f_\alpha^{-1}(U_j') \ar[r, "f_\alpha"'] & U_j' \cap f(W_\alpha) \ar[r, "F_j'"', "\cong"] & V_{j\alpha}'/\Gamma_j'
    \end{tikzcd}
  \end{equation*}

 For the rest of this proof, we replace $\Gamma^X$ with simply $\Gamma$. We will
  \begin{itemize}
  \item find a bibundle $Q_0:\Gamma(\cl A) \to \Gamma(f^*\cl{A}')$, and show it is invertible,
  \item establish a functor $\iota:\Gamma(f^*\cl{A}') \to \Gamma(\cl{A}')$ satisfying the assumptions of Lemma \ref{lem:4}, and show it is the identity if $f$ is a diffeomorphism,
  \item show that $Q := \langle \iota \rangle \circ Q_0$ is the desired bibundle.
  \end{itemize}
  The disjoint union $\cl{A} \sqcup f^*\cl{A}'$ is another atlas of $X$. Then, $\diffloc^{\cl{A} \sqcup f^*\cl{A}'}(V_{j\alpha}', V_i)$ is defined in Definition \ref{def:21} as the set of transitions $\ti{\phi}$ taking $V_{j\alpha}' \to V_i$ such that $f_\alpha^{-1} \Phi_j' = \Phi_i \circ \ti{\phi}$. Define\footnote{we are abusing notation: the symbols $P_0$ and $P$ no longer refer to pseudogroups, as they once did.} 
  \begin{equation*}
    Q_0 := \{\germ_{x'} \ti{\phi} \mid \ti{\phi} \in \diffloc^{\cl{A} \sqcup f^*\cl{A}'}(V_{j\alpha}', V_i) \text{ for some } i,j,\alpha, \text{ and } x' \in \dom \ti{\phi}\}.
  \end{equation*}
  This is an open subset, hence a submanifold, of $\Gamma(\cl{A} \sqcup f^*\cl{A}')$: around each $\germ_{x'}\ti \phi \in Q_0$, we can fit the open set $\{\germ_{x'} \ti{\phi} \mid x' \in \dom \ti{\phi}\}$. Let $\Gamma(\cl{A})$ act on $Q_0$ from the left, with anchor map the restriction of the target map $t$ of $\Gamma(\cl{A} \sqcup f^*\cl{A}')$ to $Q_0$, and multiplication given by composition of germs. Note that composition of germs is smooth by Remark \ref{prop:3}. Similarly, let $\Gamma(f^*\cl{A}')$ act on $Q_0$ from the right, along the anchor given by the source map $s$. In a diagram,
  \begin{equation*}
    \begin{tikzcd}
      & \Gamma(\cl{A}) \tensor[_s]{\times}{_{a_L}} Q_0 \ar[dl] \ar[dr, "\text{mult}"] &  & Q_0 \tensor[_{a_R}]{\times}{_t} \Gamma(f^*\cl{A}') \ar[dl, "\text{mult}"'] \ar[dr] &  \\
      \Gamma(A) \ar[dr, "t"]  &  & Q_0 \ar[dl, "t"'] \ar[dr, "s"] &  & \Gamma(f^*\cl{A}') \ar[dl, "s"'] \\
      & \bigsqcup V_i & & \bigsqcup V_{j\alpha}' .&
    \end{tikzcd}
  \end{equation*}
Because composition of germs is associative, the actions commute. The anchor maps $t$ and $s$ are also right and left invariant, respectively. It remains to show that $Q_0$ is an invertible bibundle. We will only show $Q_0 \xar{t} \bigsqcup V_i$ is right $\Gamma(f^*\cl{A}')$-principal; it is similar to show that $Q_0 \xar{s} \bigsqcup V_{j\alpha}'$ is left $\Gamma(\cl{A})$-principal.

Fix $x_0 \in V_i$, and $\tilde{\phi_0}$ a transition in $\Diff_{\text{loc}}^{\cl{A} \sqcup f^*\cl{A}'}(V_{j\alpha}',V_i)$ with $x_0$ in its image. Set $\cl{U} := \dom \tilde{\phi_0}^{-1}$, and take the section of $t$ given by
\begin{equation*}
  \sigma: \cl{U} \to P_0, \quad x \mapsto \germ_{\tilde{\phi_0}^{-1}(x)}\tilde{\phi_0}.
\end{equation*}
Define the map $\Phi$, indicated in the diagram below,
\begin{equation*}
  \begin{tikzcd}
    t^{-1}(\cl{U}) \ar[r, "\Phi"] \ar[d, "t"]& \cl{U} \fiber{s\sigma}{t} \Gamma(f^*\cl{A}') \ar[dl, "\pr_1"] \\
    \cl{U} &
  \end{tikzcd}
\end{equation*}
by
\begin{equation*}
  \Phi(\germ_{x'}\tilde{\phi}) := (\tilde{\phi}(x'), \germ_{x'} \tilde{\phi_0}^{-1}\tilde{\phi}), \quad \Phi^{-1}(x, \germ_{x'}\psi) = \germ_{x'}\tilde{\phi_0}\psi
\end{equation*}
By following the definitions, one can check that $\Phi$ is a smooth equivariant map with smooth inverse $\Phi^{-1}$. Therefore we have defined the invertible bibundle $Q_0$.

  Now, we define $\iota:\Gamma(f^*\cl{A}') \to \Gamma(\cl{A}')$. Smoothly map $\bigsqcup V_{j\alpha}'$ into $\bigsqcup V_j'$ by viewing $V_{j\alpha}' \subseteq V_j'$ and $V_{k\beta}' \subseteq V_k'$. This identification also lets us map $\diffloc^{\cl{A} \sqcup f^*\cl{A}'}(V_{j\alpha}', V_{k\beta}')$ into $\diffloc^{\cl{A}'}(V_j', V_k')$. We land in $\diffloc^{\cl{A}'}(V_j', V_k')$ because
  \begin{equation*}
    f_\beta^{-1}\Phi_k\phi' = f_\alpha^{-1}\Phi_j' \implies \Phi_k\phi' = \Phi_j', \text{ by applying } f.
  \end{equation*}
  This defines the functor $\iota:\Gamma(f^*(\cl{A}')) \to \Gamma(\cl{A}')$. On $V_{j\alpha}'$, under the identification $V_{j\alpha}' \subseteq V_j'$, this functor restricts to the inclusion of open subsets $\Gamma(\diffloc^{\cl{A}'}(V_{j\alpha}')) \to \Gamma(\diffloc^{\cl{A}'}(V_j'))$. Hence $\iota$ satisfies the conditions of Lemma \ref{lem:4}. Note that, if $f$ is a diffeomorphism, then we can take our cover by the $W_\alpha$ to consist of a single $W_\alpha = X$, and $\iota$ is the identity on the base and arrows.

  Finally, we show $Q:= \langle \iota \rangle \circ Q_0$ descends to $f$. Consider the commutative diagram from Proposition \ref{prop:1} augmented by the functor $\iota$.
  \begin{equation*}
    \begin{tikzcd}
      &  & Q_0 \ar[dl, "a_L = t"'] \ar[dr, "a_R = s"] &  & & \\
      & \bigsqcup V_i \ar[dl, "\bigsqcup \Phi_i"'] \ar[d, "\pi"] &  & \bigsqcup V_{j\alpha}' \ar[r, "\iota"] \ar[d] & \bigsqcup V_j' \ar[dr, "\bigsqcup \Phi_j'"] \ar[d, "\pi'"'] & \\
      X & \bigsqcup V_i / \Gamma(\cl{A}) \ar[l, "\ol{\Phi}"] \ar[rr, "\ol{P_0}"] &  & \bigsqcup V_{j\alpha}' / \Gamma(f^*\cl{A}') \ar[r, "\ol{\iota}"] & \bigsqcup V_j'/\Gamma(\cl{A}') \ar[r, "\ol{\Phi'}"'] & X'.
    \end{tikzcd}
  \end{equation*}
  Note that $\ol{\iota}$ is precisely the map induced by the bibundle $\langle \iota \rangle$. Fix arbitrary $u \in X$, and say $u = \Phi_i(x)$, where $x \in V_i$.  Now take $\germ_{x'}\ti{\phi} \in Q_0$ such that $\ti{\phi}(x') = t(\germ_{x'}\ti{\phi}) = x$, where we say $\ti{\phi} \in \diffloc^{\cl{A} \sqcup f^*\cl{A}'}(V_{j\alpha}', V_i)$. Then by commutativity and the definition of $\diffloc^{\cl{A} \sqcup f^*\cl{A}'}(V_{j\alpha}', V_i)$, the composition along the bottom equals
  \begin{equation*}
    \Phi_j'(x') = f_\alpha (f_\alpha^{-1}\Phi_j')(x') = f_\alpha(\Phi_i\ti{\phi})(x') = f_\alpha(\Phi_i(x)) = f_\alpha(u) = f(u)
  \end{equation*}
  Therefore $Q := \langle \iota \rangle \circ Q_0$ is a bibundle for which $\ol{\Phi'} \circ \ol{Q} \circ \ol{\Phi} = f$. Moreover, $Q_0$ is invertible, and $\langle \iota \rangle$ is locally invertible by Lemma \ref{lem:4}, hence $Q$ is locally invertible. If $f$ is a diffeomorphism, then $\langle \iota \rangle = \langle \id_{\Gamma(\cl{A}')} \rangle$, which is invertible, hence $Q$ is invertible.
\end{proof}

Now we carry out step (ii) of our plan. First we show that effective quasifold groupoids are isomorphic to germ groupoids of a special type.

Now we treat the covering groupoid $\fk{G}$.

\begin{definition}[A covering groupoid]
  \label{def:25}
  Fix a quasifold groupoid $G$, with countable quasifold atlas $\cl{A} = \{F_i:G|_{U_i} \to \Gamma_i \ltimes V_i\}$. Denote by $\fk{G}$ the pullback of $G$ by $\bigsqcup F_i^{-1}:\bigsqcup V_i \to G_0$. Its base space is $\bigsqcup V_i$, and the arrows between $x \in V_i$ and $y \in V_j$ are the arrows from $F_i^{-1}(x) \to F_j^{-1}(y)$.
\end{definition}

\begin{proposition}
  \label{prop:11}
  If $G$ is an effective quasifold groupoid, equipped with quasifold atlas $\cl{A}$, there is an isomorphism $\cl{F}:\fk{G} \cong \Gamma^{G/G_0}(\ol{\cl{A}})$, where $\ol{\cl{A}} = \{\ol{F_i}:\ol{U_i} \to \Gamma_i/V_i \mid F \in \cl{A}\}$ is a quasifold atlas of $G/G_0$. This isomorphism is identity on the base, and descends to the identity on the quotient spaces.
\end{proposition}

\begin{proof}
Given an arrow $\fk{g}$ from $x \in V_i$ to $y \in V_j$ in $\fk{G}$, take the corresponding arrow $g:F_i^{-1}(x) \to F_j^{-1}(y)$ in $G$. The effect of $g$ is an element of $\Gamma^G$ by Lemma \ref{lem:9}, so we get $\ti \phi \in \diffloc^G(G_0)$ such that $\Eff(g) = \germ_{F_i^{-1}(x)}\ti \phi$. Necessarily $\ti \phi(F_i^{-1}(x)) = F_j^{-1}(y)$. Perhaps shrinking its domain, assume $\ti \phi:W_i \to W_j$, where $W_i \subseteq V_i$ and $W_j \subseteq V_j$. Because we assume $G$ is effective, $\ti \phi$ is uniquely determined up to germ by $g$, hence by $\fk{g}$.

Following the arrows in the commutative diagram below, we see that $\phi:= F_j\ti \phi F_i^{-1} \in \diffloc^X(V_i, V_j)$.
\begin{equation*}
  \begin{tikzcd}
    F_i(W_i) \ar[r, "F_i^{-1}"] \ar[d, "\pi_i"']  \ar[dr, "\Psi_i"] & W_i \ar[r, "\ti \phi"] \ar[d, "\pi"']& W_j \ar[r, "F_j"] \ar[d, "\pi"] & F_j(W_j) \ar[d, "\pi_j"] \ar[dl, "\Psi_j"'] \\
    \pi_i(F_i(W_i)) \ar[r, "(\ol{F_i})^{-1}"'] & \pi(W_i) \ar[r, equals] & \pi(W_j) & \pi_j(F_j(W_j)) \ar[l, "(\ol{F_j})^{-1}"]
  \end{tikzcd}
\end{equation*}
Note the centre square commutes because $\ti \phi$ preserves $G$-orbits by assumption. Therefore, we define $\cl{F}$ to be identity on the base, and on arrows
\begin{equation*}
  \cl{F}(\fk{g}) := \germ_x\phi.
\end{equation*}

  By construction, for $\fk{g}$ from $x \in V_i$ to $y \in V_j$,
  \begin{align*}
    s(\cl{F}(\fk{g})) &= s(\germ_x \phi) = x \\ t(\cl{F}(\fk{g})) &= t(\germ_x\phi) = \phi(x) = F_j \ti \phi F_i^{-1}(x) = F_j F_j^{-1}(y) = y.
  \end{align*}
  So $\cl F$ commutes with the source and target. In a diagram, the first equality reads
  \begin{equation*}
    \begin{tikzcd}
      \fk{G} \ar[r, "\cl F"] \ar[dr, "s"] & \Gamma(\cl{A}) \ar[d, "s"] \\
      & \bigsqcup V_i
    \end{tikzcd}
  \end{equation*}
  and because the downward arrows are local diffeomorphisms, so is $\cl{F}$.  To check $\cl{F}$ is a homomorphism, first write (for $\fk{g}$ from $x \in V_i$ to $y \in V_j$ and $\fk{g}'$ from $y \in V_j$ to $z \in V_k$),
  \begin{equation*}
    \cl F(\fk{g}') \cl F(\fk{g}) = \germ_y F_k \ti \phi' F_j^{-1} \germ_x F_j \ti \phi F_i^{-1} = \germ_x F_k \ti \phi' \ti \phi F_i^{-1}.
  \end{equation*}
  This equals $\cl F(\fk{g}' \fk{g})$ if $\germ_x \ti \phi' \ti \phi = \Eff(g'g)$, which is true because $\Eff$ is a functor:
  \begin{equation*}
    \Eff(g'g) = \Eff(g') \Eff(g) = \germ_{F_j^{-1}(y)}\ti \phi' \germ_{F_i^{-1}(x)}\ti \phi = \germ_{F_i^{-1}(x)} \ti \phi' \ti \phi.
  \end{equation*}

  Finally, we give an inverse of $\cl{F}$. For $\germ_x\phi$ with $\phi \in \diffloc^X(V_i, V_j)$, we have $\ti \phi := F_j^{-1}\phi F_i \in \diffloc^G(G_0)$, and so $\germ_{F_i^{-1}(x)} \ti \phi \in \Gamma^G$. By Lemma \ref{lem:9}, we have $\Gamma^G = \Eff(G) \cong G$, so there is a unique\footnote{Here is where the effective assumption is necessary. If we do not assume $G$ is effective, the $g$ here will exist but may not be unique, and so our inverse would not be well-defined.} $g:F_i^{-1}(x) \to F_j^{-1}(y)$ in $G$ such that $\Eff(g) = \germ_{F_i^{-1}(x)} \ti \phi$. Set $\fk{g}$ to be the arrow from $x \in V_i$ to $y \in V_j$ corresponding to $g$. Then the map $\germ_x\phi \mapsto \fk{g}$, together with the identity on the base, is the desired inverse of $\cl{F}$.
\end{proof}

The main result of this section is now a corollary.

\begin{corollary}
\label{cor:4}
  If $G, H$ are effective quasifold Lie groupoids and $f:G/G_0 \to H/H_0$ is a local diffeomorphism, there is a locally invertible bibundle $P:G \to H$ such that $\ol{P} = f$. Moreover, if $f$ is a diffeomorphism, this bibundle is invertible.
\end{corollary}

\begin{proof}
  Fix countable atlas $\cl{A} = \{F_i:G|_{U_i} \to \Gamma_i \ltimes V_i\}$ of $G$. Then $\ol{\cl{A}} = \{\ol{F_i}:U_i/G \to V_i/\Gamma_i\}$ is a countable atlas of $G/G_0$ making it a diffeological quasifold, by Proposition \ref{prop:2}. The corresponding diffeomorphism $\ol{\Phi}:\bigsqcup V_i/\Gamma^{G_0/G}(\ol{\cl{A}}) \to G/G_0$ from Lemma \ref{lem:6} coincides with the diffeomorphism induced from the invertible bibundle $\left \langle \bigsqcup F_i^{-1} \right \rangle :\fk{G} \to G$, by the uniqueness clause from Proposition \ref{prop:1}. Here we are using Proposition \ref{prop:11} to identify $\fk{G} \cong \Gamma^{G_0/G}(\ol{\cl{A}})$.

  We similarly equip $H_0/H$ with an atlas $\ol{\cl{A}'}$.  By Proposition \ref{prop:10} applied to $f:G_0/G \to H_0/H$, we can choose a locally invertible bibundle $Q:\Gamma^{G_0/G}(\ol{\cl{A}}) \to \Gamma^{H_0/H}(\ol{\cl{A}'})$ such that $f = \ol{\Phi} \circ \ol{Q} \circ (\ol{\Phi})^{-1}$. On the other hand, the bibundle $P := \left(\left \langle \bigsqcup (F_j')^{-1} \right \rangle \circ Q \right) \circ \left \langle \bigsqcup F_i^{-1}\right \rangle^{-1}$ taking $G$ to $H$ induces precisely the map $\ol{\Phi'} \circ \ol{Q} \circ (\ol{\Phi})^{-1}$ by Proposition \ref{prop:1}, which is $f$. If $f$ is a diffeomorphism, then $Q$, hence $P$, is invertible.
\end{proof}

\newpage
%%%%%%%%%%%%%%%%%%%%%%%%%%%%%%%%%%%%%%%%%%%%%%%%%%%%%%%%%%%%%%%%
\smallskip\noindent\textbf{The germ diffeology vs the usual smooth structure}
%%%%%%%%%%%%%%%%%%%%%%%%%%%%%%%%%%%%%%%%%%%%%%%%%%%%%%%%%%%%%%%%

\section{The Germ Functional Diffeology}
\label{sec:germ-funct-diff}

For a pseudogroup $P$, there is also a canonical diffeology on $\Gamma(P)$, which is used in \cite{IZL2018} and \cite{IZP2020}. We define it in the category of diffeological spaces. On $C_{\text{loc}}^\infty(X, X')$, we may define a diffeology consisting of those parametrizations $\cl{O} \ni r \mapsto f_r$ such that
\begin{equation*}
  \{(r, x) \mid x \in \dom(f_r)\} \to X', \quad (r, x) \mapsto f_r(x)
\end{equation*}
is smooth. In particular, we require its domain to be an open subset of $\cl{O} \times X$. This is the \emph{functional diffeology} on $C_{\text{loc}}^\infty(X, X')$. 

We equip $P$ with the \emph{pseudogroup functional diffeology} consisting of parametrizations $p$ such that both $p$ and $p^{-1}$ are plots in the functional diffeology on $C_{\text{loc}}^\infty(X, X)$. Set $G:= \{(\phi, x) \mid \phi \in P, \ x \in \dom \phi\}$, and equip it with the subset diffeology induced by its inclusion in $P \times X$. Then we have the map
\begin{equation*}
  \germ: G \to \Gamma(P), \quad (\phi, x) \mapsto \germ_x\phi.
\end{equation*}
The canonical diffeology $\sr D_{\text{can}}$ on $\Gamma(P)$ is the pushforward by $\germ$ of the diffeology on $G$. We call $\sr D_{\text{can}}$ the \emph{germ functional diffeology} on $\Gamma(P)$.

  The manifold structure $\sr{A}$ on $\Gamma(P)$ introduced in Remark \ref{prop:3} generates a diffeology $\sr D_{\sr{A}}$, given by the smooth maps from Euclidean spaces into $\Gamma(P)$. Given $\Phi = (x \mapsto \germ_x\phi) \in \sr{A}$, lifting it to $x \mapsto (\phi, x)$ shows that $\Phi$ is a plot of the functional diffeology $\sr D_{\text{can}}$. So $\sr D_{\cl A} \subseteq \sr D_{\text{can}}$. However, the converse fails in general.

\begin{example}
  \label{ex:10}
  Take $M = \R^2$ and $P = \diffloc(\R^2)$. The groupoid $\Gamma(P)$ is an example of the \emph{Haefliger groupoid}. Consider the parametrization $p$ given by $(-\pi, \pi) \ni r \mapsto \germ_{\bm{0}} p_r$, where $p_r$ is counterclockwise rotation by $r$ radians. This lifts to the map
  \begin{equation*}
    (-\pi, \pi) \to G, \quad r \mapsto (p_r, \bm 0),
  \end{equation*}
  and the maps $r \mapsto \bm 0$ and $r \mapsto p_r$ are smooth (the latter in the pseudogroup functional diffeology on $P$). Therefore $r \mapsto \germ_{\bm 0}p_r$ is a plot of $\sr D_{\text{can}}$.

  On the other hand, consider the chart $\Phi \in \sr{A}$ induced by $p_0 = \id:M \to M$. Observe that
  \begin{align*}
    p^{-1}\Phi(M) &= \{r \in (-\pi, \pi) \mid \exists x \in M \text{ s.t. } \germ_{\bm 0} p_r = \germ_x\id\} \\
                  &= \{r \in (-\pi ,\pi) \mid \germ_{\bm 0}p_r = \germ_{\bm 0} \id\} \\
    &= \{0\}.
  \end{align*}
  This is not open, so the map $p$ cannot be a smooth map into $\Gamma(P)$ with manifold structure given by $\sr{A}$.
  \eoe
\end{example}

\newpage
%%%%%%%%%%%%%%%%%%%%%%%%%%%%%%%%%%%%%%%%%%%%%%%%%%%%%%%%%%%%%%%%
\smallskip\noindent\textbf{Direct proof something is not a quasifold groupoid}
%%%%%%%%%%%%%%%%%%%%%%%%%%%%%%%%%%%%%%%%%%%%%%%%%%%%%%%%%%%%%%%%

 For a contradiction assume $\Z \ltimes \R$ is a quasifold groupoid. Then its quotient space $X := \R / \Z$ is a diffeological quasifold. Define
  \begin{equation*}
    f:\R \to \R, \quad f(x) :=
    \begin{cases}
      \Psi^1(x) &\quad x \geq 0 \\
      \Psi^{-1}(x) &\quad x \leq 0.
    \end{cases}
  \end{equation*}
  Because $h$ is flat at $0$ and positive elsewhere, $f$ is smooth. For the same reason, $f^{-1}$ is smooth, so $f$ is diffeomorphism. It preserves $\Z$-orbits, so induces the identity on the quotient $X$. It is also $\Z$-equivariant, because $\Psi^1(x)$ has the same sign as $x$, and therefore induces functor $\Z \ltimes \R$ to itself, which we also denote by $f$.

  Because the functor $f$ induces the identity on $X \to X$, and so does the functor $\id_G$, by Theorem \ref{thm:1} (really Proposition \ref{prop:6}), there is a smooth natural transformation $\alpha: \id_G \implies f$. The requirements that
  \begin{equation*}
    s(\alpha(x)) = \id_G(x) = x, \quad t(\alpha(x)) = f(x)
  \end{equation*}
  imply that $\alpha(x) = (k_x,x)$, where $k_x \cdot x = f(x)$. Because the $\Z$ action is free outside of $0$, we see that
  \begin{equation*}
    k_x =
    \begin{cases}
      1 &\text{if } x > 0\\
      -1 &\text{if } x < 0.
    \end{cases}
  \end{equation*}
  Notice that $x \mapsto k_x$ cannot be smooth at $x=0$, thus $\alpha$ is not smooth at $0$, which is a contradiction.
\newpage
%%%%%%%%%%%%%%%%%%%%%%%%%%%%%%%%%%%%%%%%%%%%%%%%%%%%%%%%%%%%%%%%
\smallskip\noindent\textbf{Proof of Lemma \ref{prop:3}}
%%%%%%%%%%%%%%%%%%%%%%%%%%%%%%%%%%%%%%%%%%%%%%%%%%%%%%%%%%%%%%%%
    For simplicity, assume $M = \R^n$. First, we show $\{\germ \phi :x \mapsto \germ_x\phi \mid \phi \in \Psi\}$ is an atlas of charts for $\Gamma(\Psi)$. It covers $\Gamma(\Psi)$ by definition. Now take $\germ \phi$ and $\germ \phi'$. Denoting the domains of $\phi$ and $\phi'$ by $U$ and $U'$, we have
  \begin{align*}
    ((\germ \phi')^{-1} \circ \germ \phi) (U) &= \{x \in U' \mid \exists y \in U \text{ s.t. }\germ_x\phi'=  \germ_y\phi \} \\
                         &= \{x \in U'\mid x \in U \text{ and }\germ_x\phi' = \germ_x \phi\}
  \end{align*}
  This is open by the definition of germ, and we can now see that the composition $(\germ \phi')^{-1} \circ \germ \phi$ is the identity, hence a smooth map, between open sets. Therefore $\{\germ \phi\}$ is indeed an atlas. With the smooth structure it determines, $\Gamma(\Psi)$ is a not-necessarily Hausdorff nor second-countable manifold of dimension equal to $\dim M$.

  About each $\germ_x\phi$, consider the chart $\germ \phi$. The coordinate representation of the source, $s \circ \germ \phi$, maps $U \to U$ and is the identity. Therefore $s$ is smooth, and through each $\germ_x\phi$ it has a smooth section $\germ \phi$. This means $s$, hence $t$, are submersions.

  Fix a pair of composable arrows $(\germ_{x'}\phi', \germ_x\phi)$. The map
  \begin{equation*}
    (\germ \phi' \circ \phi, \germ \phi): U \cap \phi^{-1}(U) \to \Gamma(P) \times \Gamma(P)
  \end{equation*}
  is a chart of $\Gamma(\Psi) \times \Gamma(\Psi)$ through our fixed pair of arrows. Its image lies in the embedded submanifold of composable arrows $\Gamma(\Psi) \tensor[_s]{\times}{_t} \Gamma(\Psi)$, so it is also a chart of this submanifold. In this chart, multiplication is given by $y \mapsto \germ_y\phi' \phi$, which is (a restriction of) a chart for $\Gamma(\Psi)$. Therefore multiplication is smooth.

  Fix an arrow $\germ_x\phi$. In the chart $\germ \phi$, the inversion is given by $y \mapsto \germ_{\phi(y)}\phi^{-1}$. In other words it is the map $\germ (\phi^{-1}) \circ \phi$. This is smooth into $\Gamma(\Psi)$, hence inversion is smooth.

  In the general case, where $M \neq \R^n$, use the atlas $\{r \mapsto \germ_{x_r}\phi \mid \phi \in \Psi\}$, where the maps $r \mapsto x_r$ range over the charts of $M$.

  \newpage
  
%%%%%%%%%%%%%%%%%%%%%%%%%%%%%%%%%%%%%%%%%%%%%%%%%%%%%%%%%%%%%%%% 
  \smallskip\noindent\textbf{On restirctions}
%%%%%%%%%%%%%%%%%%%%%%%%%%%%%%%%%%%%%%%%%%%%%%%%%%%%%%%%%%%%%%%%

  Let $\iota:U \to G_0$ be the inclusion. This induces a functor $\iota:G|_U \to G$, which is the inclusion on objects and arrows. This, in turn, induces a bibundle $\langle \iota \rangle:G|_U \to G$.

  \begin{lemma}
    The restriction $P|_U$ is isomorphic to $P\circ \langle \iota \rangle$,
  \end{lemma}

  \begin{proof}
    Recall that $\langle \iota \rangle$ is $\iota_0^*(G \xar{t} G_0)$. This is
    \begin{itemize}
    \item the manifold $U \fiber{\iota}{t} G$
    \item the right anchor $s \circ \pr_2$, and the right action $(x,g) \cdot g' = (x, gg')$
      \item the left anchor $\pr_1$, and the left action $g' \cdot (x,g) = (t(g'), g'g)$,
    \end{itemize}
    We can identify this with:
    \begin{itemize}
    \item the manifold $t^{-1}(U)$
    \item the right anchor $s$, and the right action $g \cdot g' = gg'$
      \item the left anchor $t$, and the left action $g' \cdot g = g'g$.
      \end{itemize}
      Note that $G|_U$ is acting on the left, so indeed $t(g'\cdot g) = t(g') \in U$ and $g'g \in t^{-1}(U)$.

      So then we identify $P \circ \langle \iota \rangle$ with $(t^{-1}(U) \fiber{s}{a} P)/G$. The $G$ action is along the anchor $s \circ \pr_1$, and is given by
      \begin{equation*}
        (g,p)\cdot g' = (g \cdot g', (g')^{-1} \cdot p).
      \end{equation*}

      The identification is given by
      \begin{equation*}
        P|_U \to (t^{-1}(U) \fiber{s}{a} P)/G, \quad p \mapsto [1_{a(p)} , p]
      \end{equation*}
  \end{proof}

\end{document}